\documentclass{amsart}
\usepackage[english]{babel}

\usepackage[usenames,dvipsnames]{xcolor}
\usepackage[bookmarksopen,bookmarksdepth=2]{hyperref}
\hypersetup{colorlinks=true,citecolor=NavyBlue,linkcolor=BrickRed,urlcolor=Green} 
\usepackage{caption}
\usepackage[mathscr]{eucal}
\usepackage{tikz}
\usepackage{tikz-cd}
\usepackage{tikzsymbols} 
\usepackage{adjustbox}
\usepackage{enumitem}
\usepackage{dsfont}
\usepackage{verbatim}
\usepackage{mathtools}
\usepackage{fancyhdr}

\usepackage{microtype}

\renewcommand{\mathbb}{\mathbf}

\usepackage{amssymb,amsmath,amsfonts,amsthm,epsfig,amscd}
\usepackage{stmaryrd}
\usepackage[all,cmtip,poly]{xy}
\usepackage{color}
\usepackage{xcolor}
\usepackage{array}   % for \newcolumntype macro
\newcolumntype{C}{>{$}c<{$}} % math-mode version of "l" column type

\newcommand{\red}{\operatorname{red}}

\newcommand{\gK}{{\mathfrak{K}}}

\newcommand{\Akfree}{{A^{\operatorname{k-free}}}}
\newcommand{\Bdist}{{B^{\operatorname{dist}}}}

\newcommand{\Bkfree}{{B^{\operatorname{k-free}}}}

\newcommand{\dist}{{\operatorname{dist}}}
\newcommand{\kfree}{{{\operatorname{k-free}}}}

\newcommand{\Rep}{\operatorname{Rep}}
\newcommand{\Isom}{\operatorname{Isom}}

\newcommand{\BT}{\operatorname{BT}}

\newcommand{\gP}{\mathfrak{P}}

 \newcommand{\sigmabar   }{\overline{\sigma}}

\theoremstyle{plain}
\newtheorem{theorem}[equation]{Theorem}
\newtheorem{thm}[equation]{Theorem}
\newtheorem{prop}[equation]{Proposition}
\newtheorem{lem}[equation]{Lemma}
\newtheorem{lemma}[equation]{Lemma}
\newtheorem{cor}[equation]{Corollary}

\theoremstyle{definition}

\newtheorem{remark}[equation]{Remark}

\newtheorem{example}[equation]{Example}
\newtheorem{defn}[equation]{Definition}

\numberwithin{equation}{section}
\numberwithin{figure}{subsection}

\newif\iffinalrun
\iffinalrun
\else
\fi

\iffinalrun
  \newcommand{\need}[1]{}
  \newcommand{\mar}[1]{}
\else
  \newcommand{\need}[1]{{\tiny *** #1}}
  \newcommand{\mar}[1]{\marginpar{\raggedright\tiny fixme #1}}
\fi

\newcommand{\C}{\CC}
\newcommand{\F}{\FF}

\newcommand{\Q}{\QQ}

\newcommand{\Z}{\ZZ}

\newcommand{\CC}{{\mathbb C}}

\newcommand{\FF}{{\mathbb F}}
\newcommand{\GG}{{\mathbb G}}

\newcommand{\QQ}{{\mathbb Q}}

\newcommand{\ZZ}{{\mathbb Z}}

\newcommand{\bZ}{\ensuremath{\mathbf{Z}}}

\newcommand{\cC}{{\mathcal C}}

\newcommand{\cE}{{\mathcal E}}

\newcommand{\cK}{{\mathcal K}}
\newcommand{\cL}{{\mathcal L}}

\newcommand{\cM}{{\mathcal M}}
\newcommand{\cN}{{\mathcal N}}
\newcommand{\cO}{{\mathcal O}}
\newcommand{\cP}{{\mathcal P}}

\newcommand{\cR}{{\mathcal R}}

\newcommand{\cX}{{\mathcal X}}
\newcommand{\cY}{{\mathcal Y}}
\newcommand{\cZ}{{\mathcal Z}}

\newcommand{\fS}{\mathfrak{S}}
\newcommand{\gE}{\mathfrak{E}}
\newcommand{\gM}{\mathfrak{M}}
\newcommand{\gN}{\mathfrak{N}}
\newcommand{\fM}{\mathfrak{M}}

\newcommand{\gS}{\mathfrak{S}}

\newcommand{\Fbar}{\overline{\F}}
\newcommand{\Qbar}{\overline{\Q}}

\newcommand{\Fp}{\F_p}

\newcommand{\Fpbar}{\Fbar_p}

\newcommand{\Zp}{\Z_p}

\newcommand{\Qp}{\Q_p}

\newcommand{\Qpbar}{\Qbar_p}

\DeclareMathOperator{\Ext}{Ext}
\DeclareMathOperator{\kExt}{ker-Ext}

\DeclareMathOperator{\Gal}{Gal}
\DeclareMathOperator{\GL}{GL}

\DeclareMathOperator{\Hom}{Hom}

\DeclareMathOperator{\Spec}{Spec}

\DeclareMathOperator{\Spf}{Spf}

\DeclareMathOperator{\Sym}{Sym}

\newcommand{\nr}{\mathrm{nr}}
\newcommand{\sep}{\mathrm{sep}}

\newcommand{\St}{\mathrm{St}}

\newcommand{\rhobar}{\overline{\rho}}

\newcommand{\into}{\hookrightarrow}

\newcommand{\longnto}[1]{\buildrel#1\over\longrightarrow}

\newcommand{\Gm}{\GG_m}

\newcommand{\dd}{\mathrm{dd}}

\newcommand{\hatEnr}{\widehat{\cE^{\nr}}}
\newcommand{\hatOEnr}{\cO_{\hatEnr}}

\newcommand{\eb}{e}
\newcommand{\fb}{f}
\newcommand{\txI}{\mathrm{I}}
\newcommand{\txII}{\mathrm{II}}
\DeclareMathOperator{\Ad}{Ad}
\DeclareMathOperator{\crys}{crys}
\begin{document}
\selectlanguage{english}
\title{Irregular loci in the Emerton-Gee stack for $\GL_2$}
\author[Bellovin]{Rebecca Bellovin}
\author[Borade]{Neelima Borade}
\author[Hilado]{Anton Hilado}
\author[Kansal]{Kalyani Kansal}
\author[Lee]{Heejong Lee}
\author[Levin]{Brandon Levin}
\author[Savitt]{David Savitt}
\author[Wiersema]{Hanneke Wiersema}

\begin{abstract}
    Let $K/\Qp$ be unramified. Inside the Emerton--Gee stack $\cX_2$, one can consider the locus of two-dimensional mod $p$  representations  of $\mathrm{Gal}(\overline{K}/K)$ having a crystalline lift with specified Hodge--Tate weights. We study the case where the Hodge--Tate weights are irregular, which is an analogue for Galois representations of the partial weight one condition for Hilbert modular forms. We prove that if the gap between each pair of  weights is bounded by $p$ (the irregular analogue of a Serre weight), then this locus is irreducible. We also establish various inclusion relations between these  loci.  %A key ingredient is an isomorphism between certain substacks of $\cX_2$ and certain stacks introduced by Caraiani, Emerton, Gee and Savitt.
 % We utilise these isomorphisms to show that the irregular loci are isomorphic to certain reduced closed substacks whose points are mod $p$ Galois representations with irregular Hodge--Tate weights, and therefore are irreducible. We define three operators  on irregular Hodge types and show the substack in the irregular case affords inclusions into closed substacks corresponding to regular Hodge--Tate weights. 
\end{abstract}

\maketitle

% % \marginpar{quick first attempt - hope at least it will be a helpful starting point}
% \marginpar{HL: can we say ``that this locus is irreducible when the gap between Hodge--Tate weights are bounded by $p$ but not all equal to $p$.'' for more precise result? (Maybe we can even remove ``but not all equal to $p$'' because we assume irregularity here. Also, it would be nice to mention the inclusion result, like saying ``Moreover, we prove various inclusion relations between the reduced crystalline loci.'' }
\section{Introduction}

\subsection{Emerton--Gee and CEGS stacks} 
Let $K$ be a finite extension of the $p$-adic numbers $\Qp$, with residue field $k$ and absolute Galois group $G_K$. Emerton and Gee \cite{EGmoduli} have pioneered the study of certain moduli stacks of $d$-dimensional representations of $G_K$. More precisely, the Emerton--Gee stack $\cX_d$ is a formal stack over $\Spf(\Zp)$ whose $A$-valued points, for each $p$-adically complete $\Zp$-algebra $A$, are the rank $d$ \'etale $(\varphi,\Gamma)$-modules with $A$-coefficients;\ in particular, the $\F$-points of $\cX_d$ for any finite extension $\F/\Fp$ are interpretable as Galois representations $\rhobar : G_K \to \GL_d(\F)$.  The book \cite{EGmoduli} gives several important applications of this construction, including the first proof that any such~$\rhobar$ has a lift to characteristic zero, and still the only proof that any such $\rhobar$ has a crystalline lift.

It is expected that the Emerton--Gee stacks will play a central role in a categorification of the $p$-adic Langlands correspondence. This expectation is described at length in the survey article \cite{EGHcategorical}. As a first indication that the stacks $\cX_d$ have some bearing on the representation theory of $p$-adic groups, Emerton and Gee establish a  bijection between the  irreducible components of the underlying reduced substack $\cX_{d,\red}$ of $\cX_d$ in the sense of \cite[Def~3.27]{EmertonFormal}, and the irreducible $\Fpbar$-representations of $\GL_d(k)$ (which are traditionally called \emph{Serre weights}). Let $\cX_{d,\red}^\sigma$ denote the component of $\cX_{d,\red}$ corresponding to the Serre weight \textbf{}$\sigma$. 

The bijection of \cite{EGmoduli} between Serre weights $\sigma$ and components $\cX^{\sigma}_{d,\red}$ is characterized by a description of a dense set of finite type points on each component of $\cX_{d,\red}$. In the rank $d=2$ case, however, more is known:\ namely there is a complete description of all  finite type points on each component of $\cX_{d,\red}$. Recall that to each $\rhobar : G_K \to \GL_2(\F)$ the work of \cite{bdj}  associates a set $W(\rhobar)$ of Serre weights. In fact this set has several descriptions, which are known to be equivalent due to the work of a number of authors \cite{blggU2,gls12,gls13,geekisin}. One such description is in terms of the existence of crystalline lifts having certain \emph{regular} Hodge--Tate weights.
Then we have the following;\ here we recall that a Serre weight is said to be \emph{Steinberg} if it is isomorphic to a twist by a character of the representation $\St := \otimes_{\kappa : k \into \Fpbar} (\Sym^{p-1} k^2) \otimes_{k,\kappa} \Fpbar$, the tensor product taken over all the embeddings of $k$ into $\Fpbar$.

\begin{theorem}[{\cite[Thm.~1.2]{cegsA}, \cite[Thm.~8.6.2]{EGmoduli}}]\label{thm:cegs-main-X} Suppose that $p >2$.
\begin{enumerate}
    \item If the weight $\sigma$ is non-Steinberg, then $\sigma \in W(\rhobar)$ if and only if $\rhobar$ lies on the component $\cX^{\sigma}_{2,\red}$.

    \item If instead $\sigma = \chi \otimes \St$ is Steinberg, then $\chi\otimes \St \in W(\rhobar)$ if and only if $\rhobar$ lies on the union $\cX^{\chi}_{2,\red} \cup \cX^{\chi \otimes \St}_{2,\red}$.
\end{enumerate}
\end{theorem}

In fact even more is true:\ the cycles $\cX^{\sigma}_{2,\red}$ (in the non-Steinberg case) and $\cX^{\chi}_{2,\red} + \cX^{\chi \otimes \St}_{2,\red}$ (in the Steinberg case) form the cycles in a ``universal'' geometric version of the Breuil--M\'ezard conjecture for potentially Barsotti--Tate representations. 

The proof of Theorem~\ref{thm:cegs-main-X} makes  essential use of another stack $\cZ^\dd$, first defined in \cite{cegsB}, whose $A$-valued points are rank $2$ \'etale $\varphi$-modules with tamely ramified descent data, and whose $\F$-points are interpretable as Galois representations $\rhobar : G_K \to \GL_d(\F)$ having tamely potentially Barsotti--Tate lifts, or equivalently Galois representations that are not tr\`es ramifi\'ee up to twist.

The irreducible components $\overline{\cZ}(\sigma)$ of the underlying reduced substack of $\cZ^\dd$ are in bijection with the \emph{non-Steinberg} Serre weights. To prove Theorem~\ref{thm:cegs-main-X} one first establishes the following analogue of Theorem~\ref{thm:cegs-main-X}(1) for the components  $\overline{\cZ}(\sigma)$.

\begin{theorem}[{\cite[Thm~1.4(1)]{cegsA}}]\label{thm:cegs-main-Z} Suppose that $p >2$. If the weight $\sigma$ is non-Steinberg, then $\sigma \in W(\rhobar)$ if and only if $\rhobar$ lies on the component $\overline{\cZ}(\sigma)$.
\end{theorem}
The authors then transfer this result from $\cZ^\dd$ to $\cX_2$ using the fact that these stacks have the same versal deformation rings.

One of the main results of this paper is that certain closed substacks of $\cX_2$ and $\cZ^{\dd}$ are in fact isomorphic (we will be more precise in a moment about exactly which substacks). Although this is wholly expected, it is quite useful, for the following reason. The existence of substacks of mod $p$ Galois representations satisfying various $p$-adic Hodge theoretic conditions is known on the side of the Emerton--Gee stacks (we will denote these stacks by the symbol $\cX$ with various decorations), \ but not on the side of the CEGS stacks (which are the stacks here denoted $\cZ$ with various decorations). On the other hand,  calculations are generally  easier on the $\cZ$ side than on the $\cX$ side. Thanks to the isomorphism between the two sides, we can pass the existence of $p$-adic Hodge theoretic loci from the $\cX$ side to the $\cZ$ side, study their properties on the $\cZ$ side, and then transfer these results back to the $\cX$ side, which is the side of greater intrinsic interest. 
% \mar{HW: clarify notation and difference between $\cX$ and $\cZ$ here? DS: better now?}

\subsection{Irregular loci}\label{sec:comparison-of-stacks}

The other main results of our paper are applications of the above method to the closed substacks of $\cX_{2,\red}$ of mod $p$ Galois representations having  certain \emph{irregular} Hodge--Tate weights;\ these substacks have positive codimension in $\cX_{2,\red}$. The condition of being irregular is the analogue for Galois representations of the partial weight one condition for Hilbert modular forms.

To discuss these results, we assume for the remainder of the paper that the extension $K/\Qp$ is unramified and write $f = [K:\Qp]$;\ there is probably no conceptual barrier that would prevent us from studying the ramified case, but the analysis  would become considerably more complicated. Let $\underline{r}$ denote the collection of integers $\{r_{\kappa,1},r_{\kappa,2}\}_{\kappa : k \into \Fpbar}$, with $0 \le r_{\kappa,1} - r_{\kappa,2} \le p$. In the terminology of Definition~\ref{def:hodge-type-names}, we say that $\underline{r}$ is a $p$-bounded Hodge type. 

Let us write $\cX^{\underline{r}}_{\red}$ for the reduced closed substack of $\cX_{2,\red}$ whose $\F$-points are representations $\rhobar : G_K \to \GL_2(\F)$ having crystalline lifts with labeled Hodge--Tate weights $\{r_{\kappa,1},r_{\kappa,2}\}$. Thanks to the results of \cite{gls12} on the weight part of Serre's conjecture, the condition that $\sigma \in W(\rhobar)$ is equivalent to a statement about the existence of crystalline lifts of~$\rho$. As a consequence, if the Hodge type $\underline{r}$ is \emph{regular}, meaning that $r_{\kappa,1} > r_{\kappa,2}$ for all~$\kappa$, then Theorem~\ref{thm:cegs-main-X} says precisely that $\cX^{\underline{r}}_{\red}$  is one of the irreducible components of $\cX_{2,\red}$ (or a union of two irreducible components, in case all the differences $r_{\kappa,1} - r_{\kappa,2}$ are equal to $p$).

We are interested in saying something about the stacks $\cX^{\underline{r}}_{\red}$ in the \emph{irregular} case, i.e., when $r_{\kappa,1} = r_{\kappa,2}$ for one or more embeddings $\kappa$ (in which case $\cX^{\underline{r}}_{\red}$ has codimension in $\cX_{2,\red}$ equal to the number of embeddings $\kappa$ such that $r_{\kappa,1} = r_{\kappa,2}$).  To explain how we do this, we need to introduce the companions of the stacks $\cX^{\underline{r}}_{\red}$ on the side of the CEGS stacks. 

By construction the stack $\cZ^\dd$ is equipped with a morphism
\begin{equation}\label{eq:Zdd-resolution}
	\cC^{\dd,\BT} \to \cZ^\dd	
\end{equation}
where  $\cC^{\dd,\BT}$ is the stack of Breuil--Kisin modules of height at most one with tame descent data and satisfying a Kottwitz-type determinant condition. Indeed, $\cZ^{\dd}$ is defined to be the scheme-theoretic image (in the sense of \cite{emertongeeproper}) of $\cC^{\dd,\BT}$ in the stack of \'etale $\varphi$-modules with descent data. 

The map \eqref{eq:Zdd-resolution} can be thought of as a partial resolution of the stack $\cZ^\dd$. Resolutions of moduli of Galois representations by moduli of objects coming from integral $p$-adic Hodge theory  have played a fundamental role in the deformation theory of Galois representations,  and hence in the study of automorphy lifting theorems, going back to the work of Wiles and Taylor--Wiles \cite{Wiles,TaylorWiles}. In this particular guise the inspiration comes from the work of Kisin, as the map \eqref{eq:Zdd-resolution} can be thought of as a globalization of the maps $\Theta_{V_{\F}}$ of \cite{kis04}.

A large part of this article can be thought of as an analysis of some of the finer properties of the map \eqref{eq:Zdd-resolution}. To say more, we need to introduce some additional notation. The stack $\cC^{\dd,\BT}$ has a decomposition
\[ \coprod_{\tau} \cC^{\tau,\BT} \]
where the disjoint union is taken over tame inertial types $\tau:I_K \to \GL_2(\Qpbar)$, and the substack $\cC^{\tau,\BT}$ consists of Breuil--Kisin modules whose descent data has type $\tau$. We write $\cZ^\tau$ for the scheme-theoretic image of $\cC^{\tau,\BT}$ in $\cZ^{\dd}$.  By Theorem~1.4(2) of \cite{cegsA} the underlying reduced substack  $\cZ^{\tau,1}$ of $\cZ^{\tau}$ is precisely the union of the components $\overline{\cZ}(\sigma)$ for Serre weights $\sigma$ that occur as Jordan--H\"older factors of $\overline{\sigma}(\tau)$, the reduction mod $p$ of the representation $\sigma(\tau)$ associated to $\tau$ by the inertial local Langlands correspondence. We can now state precisely which of the Emerton--Gee and CEGS stacks we check are isomorphic.

\begin{thm}[{Theorem~\ref{thm:EG-CEGS-isom}}]
    The stack $\cZ^{\tau}$ is isomorphic to $\cX^{\tau,\BT}$, the Emerton--Gee stack of potentially Barsotti--Tate representations of type $\tau$.
\end{thm}

Assume henceforth that the type $\tau$ is non-scalar. The underlying reduced substack $\cC^{\tau,\BT,1}$ of $\cC^{\tau,\BT}$ has precisely $2^f$ irreducible components $\cC^\tau(J)$, indexed by subsets~$J$ of the set of embeddings $\kappa: k \into \Fpbar$. Write $\cZ^\tau(J)$ for the scheme-theoretic image of $\cC^\tau(J)$ in $\cZ^{\tau,1}$.

There is a combinatorial formula (which we  recall in equation \eqref{def:s_J} below) which associates to each pair $(\tau,J)$ 
a tuple of integers $s_{J,\kappa} \in [-1,p-1]$ indexed by the embeddings $\kappa: k \into \Fpbar$, as well as a character $\Theta_J : k \to \Fpbar^\times$.  Define $\cP_\tau$ to be the collection of sets $J$ such that  $s_{J,\kappa} \in [0,p-1]$ for all $\kappa$, i.e., such that $s_{J,\kappa} = -1$ does not occur. Then for each $J \in \cP_\tau$, it makes sense to define the Serre weight
\[ \sigmabar(\tau)_J = (\Theta_J \circ \det) \otimes \bigotimes_{\kappa : k \into \Fpbar} (\Sym^{s_{J,\kappa}} k^2) \otimes_{k,\kappa}\Fpbar.
\]
Every irreducible component of both $\cC^{\tau,\BT,1}$ and $\cZ^{\tau,1}$  has dimension $[K:\Qp]$. Under the map
\[ \cC^{\tau,\BT,1} \to \cZ^{\tau,1}, \]
it is proved in \cite{cegsC} that
\begin{itemize}
    \item If $J \in \cP_\tau$, then the component $\cC^\tau(J)$ dominates a component of $\cZ^{\tau,1}$;\ that is, $\cZ^\tau(J)$ is some irreducible component of $\cZ^{\tau,1}$; while on the other hand,
    \item If $J \not\in \cP_\tau$, then the scheme-theoretic image of $\cC^\tau(J)$ in $\cZ^{\tau,1}$ has positive codimension. Following \cite{cegsC} we refer to these $\cC^{\tau}(J)$ as ``vertical components'' of $\cC^{\tau,\BT,1}$.
\end{itemize}
The first part is made more precise in \cite{cegsA}, as follows.
\begin{thm}[{\cite[Thm.~6.2(5)]{cegsA}}]\label{thm:cegs-z-tau-j}
  If $J \in \cP_{\tau}$, then $\cZ^\tau(J)$ is precisely the component $\overline{\cZ}(\sigmabar(\tau)_J)$;\ in particular, we have $\rhobar \in \cZ^\tau(J)$ if and only if $\sigmabar(\tau)_J \in W(\rhobar)$, and $\cZ^{\tau}(J)$ depends only in the Serre weight $\sigmabar(\tau)_J$.
\end{thm}

Equivalently, again using the results of \cite{gls12} on the weight part of Serre's conjecture,  Theorem~\ref{thm:cegs-z-tau-j} can be rephrased as follows. (Here  $\tilde{\Theta}_J$ is any extension to $G_K$ of the inertial character identified with $\Theta_J$  via Artin reciprocity. See Section~\ref{sec:notation:HT} for the various normalizations related to Serre weights and Hodge--Tate weights that we use throughout this paper.)

\begin{thm}[{{\cite[Thm.~6.2(5)]{cegsA}}}, second version\textbf{}]\label{thm:cegs-z-tau-j-HT-wets}
  Suppose $J \textbf{}\in \cP_\tau$. Then $\rhobar \in \cZ^\tau(J)$ if and only if $\rhobar \otimes \tilde{\Theta}_J^{-1}$ has a crystalline lift with $\kappa$-labeled Hodge--Tate weights $\{-s_{J,\kappa},1\}$ at each embedding $\kappa : k \into \Fpbar$.
\end{thm}

In other words, for regular Hodge types $\underline{r}$, the companions of the stacks $\cX^{\underline{r}}_{\red}$ on the CEGS side are precisely the stacks $\cZ^{\tau}(J)$ for $J \in \cP_{\tau}$. It is natural, then, to imagine that the  positive codimension loci $\cZ^{\tau}(J)$ for $J\notin \cP_\tau$ are the companions of the stacks $\cX^{\underline{r}}_{\red}$ for irregular $\underline{r}$;\ and indeed, this is what we show.

Observe that the statement of Theorem~\ref{thm:cegs-z-tau-j-HT-wets}  makes sense even if some $s_{J,\kappa}$ is equal to $-1$, i.e., if $J \not\in \cP_{\tau}$. The only difference is that when $J \in \cP_{\tau}$, the  Hodge--Tate weights  $\{-s_{J,\kappa},1\}$ are always regular (distinct), while if $s_{J,\kappa} = -1$ then the  $\kappa$-labeled Hodge--Tate weights $\{-s_{J,\kappa},1\}$ are irregular. We prove that the statement of Theorem~\ref{thm:cegs-z-tau-j-HT-wets} remains valid even if $J \not\in \cP_\tau$;\ that is, we prove the following.

\begin{thm}[{Theorem~\ref{thm:irregular-locus}}]\label{thm:main-result-1-intro}
   For general $J$ we have $\rhobar \in \cZ^\tau(J)$ if and only if $\rhobar \otimes \tilde{\Theta}_J^{-1}$ has a crystalline lift with $\kappa$-labeled Hodge--Tate weights $\{-s_{J,\kappa},1\}$ at each embedding $\kappa : k \into \Fpbar$. In particular $\cZ^{\tau}(J)$ depends only on the $s_{J,\kappa}$ and on $\Theta_J$;\ or, if one likes, only on the expression for the ``fake Serre weight'' $\sigmabar(\tau)_J$, whose definition contains a $\Sym^{-1}$ if $J \not\in \cP_{\tau}$.
\end{thm}

As an application, we  deduce the following.  

\begin{cor}[{Corollary~\ref{cor:irreducibility-of-EG-irregular}}]\label{cor:intro}  Suppose that $0 \le r_{\kappa,1} - r_{\kappa,2} \le p$ for all $\kappa$, with not all differences equal to $p$.
The closed substack $\cX_{\red}^{\underline{r}}$ of $\cX_{2,\red}$ whose finite type points are representations $\rhobar:G_K \to \GL_2(\Fpbar)$ having crystalline lifts with labeled Hodge--Tate weights $\{r_{\kappa,1},r_{\kappa,2}\}$ is irreducible.
\end{cor}

This result is new in the irregular case. The point is that by Theorem~\ref{thm:main-result-1-intro} this locus is isomorphic to one of the stacks $\cZ^{\tau}(J)$, and the latter is irreducible by construction.

Since the closed substack $\cX^{\underline{r}}_{\red}$ in the irregular case has positive codimension in $\cX_{2,\red}$, it is reasonable to expect that there are inclusions $\cX^{\underline{r}}_{\red} \subset \cX^{\underline{r}'}_{\red}$ for other collections of Hodge--Tate weights $\underline{r}'$. In the final section of the paper, we establish this expectation in a number of cases. 
We define three operators $\theta_{\kappa}$, $\mu_{\kappa}$, and $\nu_{\kappa}$ on  Hodge types $\underline{r}$ that are irregular at $\kappa$ (see Definition~\ref{def:operators};\ the first two can be viewed as analogues of partial theta operators and partial Hasse invariants, as described in the work of Diamond and Sasaki on geometric Serre weight conjectures (\cite{DiamondSasaki}) and discussed further in \cite{HWThesis}. We then establish the following.

\begin{thm}[{Theorem~\ref{thm:operator-lifts-exist}}]\label{thm:operators-intro}
    Suppose  $\underline{r}' \in \{ \theta_{\kappa}(\underline{r}), \mu_{\kappa}(\underline{r}), \nu_{\kappa}(\underline{r}) \}$ and assume that~$\underline{r}'$ remains $p$-bounded. Then $\cX^{\underline{r}}_{\red} \subset \cX^{\underline{r}'}_{\red}$.
\end{thm}

As a consequence  we deduce that every representation with a crystalline lift of Hodge type $\underline{r}$ also has a crystalline lift of Hodge type $\underline{r}'$, for $\underline{r},\underline{r}'$ as in the theorem.

\subsection{Outline of the paper} We begin in Section~\ref{sec:cegs-review} by recalling from \cite{cegsB} the definitions of various stacks of Breuil--Kisin modules and \'etale $\varphi$-modules with descent data, and reviewing many of the results from the papers \cite{cegsA,cegsB,cegsC} that we will need.

In Section~\ref{sec:shape} we analyze the irreducible components of the tamely potentially Barsotti--Tate Breuil--Kisin moduli stacks $\cC^{\tau,\BT,1}$ from the point of view of the \emph{shape} of a Breuil--Kisin module, as studied in \cite{breuillatticconj,CDMa,LLLM}. Using this idea we give a new description of the irreducible components of $\cC^{\tau,\BT,1}$ with the advantage that we can characterize all of the $\Fpbar$-points on each component $\cC^\tau(J)$ of $\cC^{\tau,\BT,1}$. This stands in contrast to \cite{cegsC}, where only a dense set of points are described. As an application, in Section~\ref{sec:finite-type-z} we  describe the $\Fpbar$-points of the stacks $\cZ^{\tau}(J)$. This is a key ingredient in the rest of the paper, and also leads to a new and purely local proof of the characterization of the irreducible components of $\cZ^{\dd}$ in terms of crystalline lifts.

In Section~\ref{sec:comparisons} we introduce the Emerton--Gee stacks to the discussion, and establish the isomorphism between the stacks $\cX^{\tau,\BT}$ and $\cZ^{\tau}$. One consequence  is the existence of a reduced closed substack $\cZ^{\underline{r}}_{\red}$ of $\cZ^{\dd}$, for each $p$-bounded and non-Steinberg Hodge type $\underline{r}$, whose $\Fpbar$-points are precisely the representations $\rhobar:G_K \to \GL_2(\Fpbar)$ having a crystalline lift of type $\underline{r}$.  In Section~\ref{sec:comparison-section} we combine the results of Section~\ref{sec:shape} with results from \cite{gls12} and combinatorial input from Section~\ref{sec:combinatorial-results} to prove that the stacks $\cZ^{\underline{r}}_{\red}$ are equal to the stacks $\cZ^{\tau}(J)$ for suitable choices of $\tau$ and $J$. This  establishes Theorem~\ref{thm:main-result-1-intro} and Corollary~\ref{cor:intro}.

Finally in Section~\ref{sec:geography} we prove Theorem~\ref{thm:operators-intro}. In fact we give two proofs. One argument is relatively direct and computational. The other is more geometric, but also somewhat more involved, making use of the description of the components of $\cC^{\tau,\BT,1}$ in terms of shape.

\subsection{Notation and conventions}\label{sec:notation}
Let $p > 2$ be a prime number. Throughout the paper we fix $K/\Qp$, an unramified extension of $\Qp$ of degree $f$ with residue field $k$. 

We also fix an algebraic closure $\Qpbar$ of $\Qp$, with residue field $\Fpbar$. Our representations of the Galois group $G_K$ will have coefficients in these fields. Let $E$ be a finite extension of $\Qp$ contained in $\Qpbar$, with ring of integers $\cO$, uniformizer $\varpi$, and residue field $\F$. As usual we will assume that the coefficient field $E$ is ``sufficiently large'', in a sense that we make precise in Section~\ref{sec:tame-types} below.

Since $K/\Qp$ is unramified we can (and do) identify the embeddings $K \into \Qpbar$ with embeddings $k \into \Fpbar$. We fix some embedding $\kappa_0 : k \into \Fpbar$ and recursively label the remaining embeddings by elements of $\Z/f\Z$ by taking $\kappa_{i+1}^p = \kappa_i$.

\subsubsection{Hodge--Tate weights and Serre weights}\label{sec:notation:HT} Throughout this paper we will use notation and conventions as in the series of papers \cite{cegsA,cegsB,cegsC}. Hodge--Tate weights are normalized so that the cyclotomic character has all Hodge--Tate weights equal to $-1$.  We normalize local class field theory so that uniformizers correspond to geometric Frobenius elements.

A \emph{Serre weight} is an irreducible $\Fpbar$-representation of $\GL_2(k)$. Each Serre weight has the form 
\[\sigmabar_{\vec{t},\vec{s}}=\otimes_{j=0}^{f-1} (\det {}^{t_j} \Sym^{s_j} k^2) \otimes_{k,\kappa_j} \Fpbar\]
for integers $t_j$ and integers $0 \leq s_j \leq p-1$. If we furthermore assume that $0 \leq t_j \leq p-1$  and not all $t_j$ are $p-1$, each Serre weight  has a unique representation as one of the  $\sigmabar_{\vec{t},\vec{s}}$'s.

To each representation $\rhobar : G_K \to \GL_2(\Fpbar)$ one associates a set $W(\rhobar)$ of Serre weights. In our conventions, we have $\sigmabar_{\vec{t},\vec{s}} \in W(\rhobar)$ if and only if $\rhobar$ has a crystalline lifts with Hodge--Tate weights $\{-s_j-t_j,1-t_j\}$ for the embedding $\kappa_j$ (\emph{cf.}~\cite[Def.~A.3]{cegsC}).

\subsubsection{Tame types.}\label{sec:tame-types} Throughout the paper we write $\tau$ for a non-scalar tame inertial type $I_K \to \GL_2(\cO)$. Such a representation is of the form $\tau
\simeq \eta \oplus \eta'$, and we say that $\tau$ is a \emph{ principal series
  type} if 
$\eta,\eta'$ both extend to characters of $G_K$. Otherwise,
$\eta'=\eta^q$, and $\eta$ extends to a character of~$G_L$, where $L$ denotes the unramified quadratic extension of $K$.
In this case we say
that~$\tau$ is a \emph{cuspidal type}. 

Throughout the paper we will often need to handle the principal series and cuspidal cases separately. 
We define
\begin{align*}
f':=\begin{cases}
f &\text{ if } \tau \text{ is principal series}\\
2f &\text{ if } \tau \text{ is cuspidal}
\end{cases}
\end{align*}
and set $e' = p^{f'} - 1$. Fix a uniformizer $\pi$ of $K$ and choose $\pi'$ such that $(\pi')^{e'} = \pi$. In the principal series case we define $K' = K(\pi')$, while in the cuspidal case we define $K' = L(\pi')$, so that in either case $K'$ is a finite tamely ramified Galois extension of~$K$ with inertial degree $f'$ and having the property that $\tau|_{I_{K'}}$ is trivial. We assume that $E$ is sufficiently large in the sense that all embeddings $K' \into \Qpbar$ have image in~$E$.

Write $k'$ for the residue field of $K'$, and let $\kappa'_0 : k' \into \Fpbar$ be any extension of $\kappa_0$  to $k'$. Recursively label the embeddings $k' \into \Fpbar$ by elements of $\Z/f'\Z$ by taking $(\kappa'_{i+1})^{p} = \kappa'_i$. For all $g \in G_K$ we set $h(g) = g(\pi')/\pi' \in \mu_{e'}(K')$. Identifying $\mu_{e'}(K')$ with $\mu_{e'}(k')$, we  can then define fundamental characters $\omega_i$ of level
$f$ by setting $\omega_i = \kappa_i \circ h : I_K \to \Fpbar^\times$  for
     each $i \in \Z/f\Z$ (\emph{cf.}~\cite[Lem.~1.4.1]{cegsC} and the discussion following). Similarly define  fundamental characters $\omega'_i$ of level
$f'$ by setting $\omega'_i = \kappa'_i \circ h : I_K \to \Fpbar^\times$  for
     each $i \in \Z/f'\Z$. Let $\widetilde{\omega}_i : I_K \to \Qpbar^{\times}$ denote the multiplicative lift of $\omega_i$, and similarly for $\widetilde{\omega}'_i$.

In the cuspidal case, let $\mathfrak{c} \in \Gal(K'/K)$ denote the unique nontrivial element fixing $\pi'$.

\subsubsection{Inertial local Langlands and Serre weights}\label{sec:ill-sw}

Henniart's appendix to \cite{breuil-mezard}
associates a finite dimensional irreducible $E$-representation $\sigma(\tau)$ of
$\GL_2(\cO_K)$ to each inertial type $\tau$; we refer to this association as the {\em
  inertial local Langlands correspondence}. The reduction modulo $p$ of $\sigma(\tau)$ is not well-defined, but its Jordan--H\"older factors are. As in \cite{cegsA} we normalize the inertial local Langlands correspondence so that $\rhobar$ has a potentially Barsotti--Tate lift of $\tau$ if and only if at
  $W(\rhobar)$ contains at least one Jordan--H\"older factor of the reduction mod $p$ of  $\sigma(\tau)$.

We now give an explicit description of the Jordan--H\"older factors of the reduction mod $p$ of  $\sigma(\tau)$, following  the recipe from \cite[Appendix~A]{cegsC}. Recall that we assume $\tau$ to be a non-scalar tame type. We define $0 \le \gamma_i \le p-1$ (for $i \in \Z/f'\Z$) 
to be the unique integers such that 
\begin{equation}\label{eq:gamma_i} \eta (\eta')^{-1} = \prod_{i=0}^{f'-1}
(\widetilde{\omega}'_{i})^{\gamma_i} .
\end{equation}
Observe in the cuspidal case  that $\gamma_i + \gamma_{i+f} = p-1$.

\begin{defn}\label{def:profile}
Let us say that a subset $J \subset \Z/f' \Z$ is a \emph{profile} if either 
\begin{itemize}
    \item $\tau$ is non-scalar principal series, and $J$ is any subset;\ or
    \item $\tau$ is cuspidal, and for all $i$ we have $i \in J$ if and only if $i+f \not\in J$.
\end{itemize}
\end{defn}
 For each profile $J$, we define tuples of integers $(s_{J, i})_i$ and $(t_{J,i})_i$ indexed by $\Z/f'\Z$, as follows. 
\begin{align}\label{def:s_J}
&s_{J, i} := \begin{cases}
p-1-\gamma_i - \delta_{J^{c}}(i) &\text{ if } i-1\in J, \\
\gamma_i - \delta_J(i) &\text{ if } i-1 \not\in J.
\end{cases}\\
&t_{J, i} := \begin{cases}
\gamma_i + \delta_{J^{c}}(i) \hspace{1.25cm}&\text{ if } i-1\in J, \\
0 \hspace{1.25cm}&\text{ if } i-1 \not\in J.
\end{cases}
\end{align}
Note that $s_{J, i} \in [-1,p-1]$, $t_{J, i} \in [0,p]$, and $s_{J, i}$ is $f$-periodic.  Define
$\Theta_J: k^{\times} \to \F^\times$ to be the unique character such that
\begin{equation}\label{def:theta_J} \Theta_J \circ \mathrm{N}_{k'/k} = \eta' \otimes \prod_{i \in \Z/f'Z} (\kappa'_i)^{t_{J, i}}.\end{equation}
Here we regard $\eta'$ as a character of $k'$ via the Artin map for $K'$, and $\mathrm{N}_{k'/k}$ denotes the norm map. (It is true, but not obvious, that the right-hand side of \eqref{def:theta_J} factors through $\mathrm{N}_{k'/k}$.)

\begin{defn}\label{defn:def-of-sigmabar_J}
Define  $S^{\tau}(J) := \{i \in \Z/f\Z \,|\, s_{J, i} = -1\}$
and take $\cP_{\tau}$ to be the set of profiles $J$ such that $S^{\tau}(J)$ is empty.

Then for each $J \in \cP_{\tau}$, we define
\[ \sigmabar(\tau)_J = ({\Theta}_J \circ \det) \otimes  \bigotimes_{i \in \Z/f\Z} (\Sym^{s_{J,i}} k^2) \otimes_{k, \kappa_i} \Fpbar.\] 
\end{defn}

\begin{thm}
The Jordan--H\"older factors of the reduction mod $p$ of $\sigma(\tau)$ are precisely the Serre weights $\sigmabar(\tau)_J$ for $J \in \cP_\tau$.
\end{thm}

\subsection{Acknowledgments} This project began at the APAW Collaborative Research Workshop at the University of Oregon in August 2022. We are grateful to Ellen Eischen, Maria Fox, Cathy Hsu, and Aaron Pollack for organizing the workshop, as well as for support from Eischen's NSF CAREER grant DMS-1751281 and her NSA MSP conference grant H98230-21-1-0029. The work was completed during various visits by six of us to the trimester program ``The Arithmetic of the Langlands Program'' at the Hausdorff Insitute for Mathematics, funded by the Deutsche Forschungsgemeinschaft under Germany's Excellence Strategy – EXC-2047/1 – 390685813. We thank Frank Calegari, Ana Caraiani, Laurent Fargues, and Peter Scholze for their efforts organizing the trimester.

We thank Matthew Emerton and Toby Gee for valuable conversations. B.L. was supported by National Science Foundation grant DMS-1952556 and the Alfred P.\ Sloan Foundation. D.S.\ was supported by NSF grant DMS-1952566. H.W.\ was supported by the Herchel Smith Postdoctoral Fellowship Fund, and the Engineering and Physical Sciences Research Council (EPSRC) grant EP/W001683/1. %\mar{others still to fill out}

\section{Stacks of Breuil--Kisin modules and \'etale \texorpdfstring{$\varphi$}{phi}-modules}\label{sec:cegs-review}
% \textcolor{purple}{HL: I think below paragraphs are too long as a intro to a section. I suggest to shorten it and move the detail to the following subsections.} \textcolor{red}{Agreed, though a different more relevant intro could be included here}

% \textcolor{green}{RB: I would say that much of this should go in the introduction section, since it sets up the questions we're considering for the entire paper.}

As we have explained in the introduction, the main objects of study in this paper are certain stacks $\cC^{\tau,\BT}$ of Breuil--Kisin modules and $\cZ^{\tau}$ of \'etale $\varphi$-modules, as well as the map 
\[ \cC^{\tau,\BT} \to \cZ^{\tau} \]
between them. We begin with a brief recollection of the definitions and basic properties of these objects.

\subsection{Breuil--Kisin modules}\label{subsec: bk modules definitions}

As in~\cite{cegsC}, we will consider Breuil--Kisin modules with coefficients and descent data.  Let $\mathfrak{S}:=W(k')[\![u]\!]$, and extend the arithmetic Frobenius on $W(k')$ (i.e. the homomorphism $\varphi:W(k')\rightarrow W(k')$ induced by $x\mapsto x^p$ on $k'$) to a self-map $\varphi$ of $\mathfrak{S}$ by setting $\varphi(u)=u^p$.  We extend the action of $\Gal(K'/K)$ on $W(k')$ to $\mathfrak{S}$ via $g(u)=h(g)u$. 

If $A$ is a $p$-adically complete $\Z_p$-algebra, we set $\mathfrak{S}_A:=(W(k') \otimes_{\Zp} A)[\![u]\!]$ and extend the actions of $\varphi$ and $\Gal(K'/K)$ $A$-linearly.  Setting $v:=u^{e(K'/K)}$, we define the subring $\mathfrak{S}_A^0:=(W(k) \otimes_{\Zp} A)[\![v]\!]$, which is preserved by $\varphi$ but on which $\Gal(K'/K)$ acts trivially.
 Let $E(u)$ denote the minimal polynomial of $\pi'$ over $W(k')$.

\begin{defn}
A \emph{Breuil--Kisin module with $A$-coefficients and descent data} is a finite projective $\mathfrak{S}_A$-module $\mathfrak{M}$ together with semilinear maps
\[  \varphi_{\mathfrak{M}}: \mathfrak{M}\rightarrow \mathfrak{M}    \]
and 
\[   \widehat{g}: \mathfrak{M}\rightarrow \mathfrak{M},\qquad g\in \Gal(K'/K).   \]
Write $\varphi^*\mathfrak{M} := \mathfrak{S} \otimes_{\varphi,\mathfrak{S}} \mathfrak{M}$.
We impose the further requirements that the linearization 
\[  \Phi_{\mathfrak{M}}: \varphi^\ast\mathfrak{M}\rightarrow \mathfrak{M}  \]
of $\varphi_{\gM}$ is an isomorphism after inverting $E(u)$, that each $\widehat{g}$ commutes with $\varphi_{\mathfrak{M}}$, and that $\widehat{g}_1\circ\widehat{g}_2=\widehat{g_1g_2}$ for all $g_1, g_2\in\Gal(K'/K)$.  We say that the Breuil--Kisin module~$\mathfrak{M}$ has \emph{height at most $h$} if the cokernel of $\Phi_{\mathfrak{M}}$ is killed by $E(u)^h$.

Morphisms of Breuil--Kisin modules are morphisms of $\mathfrak{S}_A$-modules that commute with the actions of $\varphi$ and $\Gal(K'/K)$.
\end{defn}

Recall from Section~\ref{sec:notation} that we have fixed a coefficient field $E/\Qp$ with ring of integers $\cO$, uniformizer $\varpi$, and residue field $\F$, and that $E$ is sufficiently large in the sense that it admits embeddings of the field $K'$.

\begin{defn}
For each embedding $\kappa'_i : k' \into \Fpbar$, which we identify with its  lift $W(k') \into \Qpbar$, there is a corresponding idempotent $e_i \in W(k') \otimes_{\Zp} \cO$ 
 such that $x \otimes 1$  and $1 \otimes \kappa'_i(x)$ have the same action on $e_i(W(k') \otimes_{\Zp} \cO)$.
 
 For any $\cO$-algebra $A$ and any $A$-module $\gM$, we set $\gM_i := e_i \gM$.  In case $\gM$ is a Breuil--Kisin module, we write 
\[ \Phi_{\mathfrak{M},i}: \varphi^\ast(\mathfrak{M}_{i-1})\rightarrow \mathfrak{M}_i\] for the  morphism induced by $\Phi_{\fM}$, which we call the $i$-th \emph{partial Frobenius} morphism.
\end{defn}

We note the following lemma, which is an immediate consequence of \cite[Prop.~5.1.9(1)]{emertongeeproper}.

\begin{lemma}\label{lem:zariski-locally-free}
    Let $A$ be a $p$-adically complete $\cO$-algebra, and let $\gM$ be a Breuil--Kisin module with $A$-coefficients and descent data. Then each $\gM_i$ is Zariski locally  on $\Spec(A)$ free as an $A[\![u]\!]$-module.
\end{lemma}

\begin{defn}
 Let $\tau$ be a tame inertial type, and let $A$ be a $p$-adically complete $\cO$-algebra. We say that the Breuil--Kisin module $\gM$ with $A$-coefficients and descent data is \emph{of type $\tau$} provided that Zariski locally on $\Spec(A)$ there is an $I(K'/K)$-equivariant isomorphism $\gM_i/u\gM_i \cong A \otimes_{\cO} \tau$ for each $i$. 
\end{defn}

\begin{defn}
 We define $\cC^{\tau,\BT}$ to be the  stack over $\Spf(\cO)$ which associates to any $\cO/\varpi^a$-algebra $A$ the groupoid of rank $2$ Breuil--Kisin modules with $A$-coefficients and descent data, of type $\tau$, and of height at most $1$, and additionally satisfying the Kottwitz-type strong determinant condition of \cite[\S4.2]{cegsB}. 
 
 We define $\cC^{\dd,\BT}$ to be the  union of the stacks $\cC^{\tau,\BT}$ for varying $\tau$ (including scalar types), and further write
$\cC^{\dd,\BT,1}$ and $\cC^{\dd,\BT,1}$ for the special fibers $\cC^{\dd,\BT} \times_{\cO} \F$ and $\cC^{\tau,\BT} \times_{\cO} \F$ respectively.
\end{defn}

\begin{remark}
We do not write out the strong determinant condition explicitly, because we will not need it. However, we recall from \cite[Lem~4.2.16]{cegsB} that this condition guarantees that the $\Spf(\cO)$-points of $\cC^{\tau,\BT}$ correspond to potentially Barsotti--Tate representations with Hodge--Tate weights $\{0,1\}$  (rather than \emph{contained in} $\{0,1\}$) at each embedding.
\end{remark}

\begin{prop}\label{prop:reduced-f-strong-det}
Suppose that $A$ is a reduced $\F$-algebra and let $\gM$ be a rank $2$ Breuil--Kisin module with $A$-coefficients and descent data, of type $\tau$, and of height at most $1$. Then $\gM$ is an object of $\cC^{\tau,\BT}(A)$ {\upshape(}i.e., it satisfies the strong determinant condition{\upshape)} if and only if locally on $\Spec(A)$ the determinant of each partial Frobenius maps $\Phi_{\gM,i}$, with respect to some {\upshape(}hence any{\upshape)} choice of bases, lies in $u^{e'} \cdot A[\![u]\!]^\times$.
\end{prop}

\begin{proof}
In case $A = \F'$ is a finite extension of $\F$, the `only if' direction is given by \cite[Lem.~4.2.11(2)]{cegsB}. The converse to \cite[Lem.~4.2.11(2)]{cegsB} is false in general, because the type $\tau$ in that reference is allowed to be \emph{mixed} (\emph{cf}.~\cite[Def.~3.3.2]{cegsB}). Since we assume here that the type is unmixed, the $A[\![v]\!]$-determinants of the maps $\Phi_{\gM,i,\xi}$ (the $\xi$-isotypic part of $\Phi_{\gM,i}$ for some character $\xi$ of $I(K'/K)$, in the notation of \cite{cegsB}) are all equal up to units, and indeed equal up to units to the $A[\![u]\!]$-determinant of $\Phi_{\gM,i}$;\ and then the argument for \cite[Lem.~4.2.11(2)]{cegsB} runs in the reverse direction as well.

For general coefficients, we reduce immediately to the case where $A$ has finite type as an $\F$-algebra. Since $A$ is then both reduced and Jacobson,  the strong determinant condition for $\gM$ is equivalent to the strong determinant condition for the base change of $\gM$ to each closed point of $A$. The result then follows from the case $\F'/\F$ finite of the previous paragraph.
\end{proof}

The following results are proved in \cite[Cor.~3.1.8, Cor.~4.5.3, Prop.~5.2.21]{cegsB} and \cite[Thm.~5.4.3]{cegsC}.

\begin{thm}\label{thm:stacks-results} The stacks $\cC^{\dd,\BT,1}$, and $\cC^{\tau,\BT,1}$ are algebraic stacks of finite type over $\cO$, while the stacks $\cC^{\dd,\BT}$ and $\cC^{\tau,\BT}$ are $\varpi$-adic formal algebraic stacks. Moreover:
\begin{enumerate}
  \item $\cC^{\tau,\BT}$ is analytically normal, Cohen--Macaulay, and flat over $\cO$.
  \item The stacks $\cC^{\dd,\BT,1}$ and $\cC^{\tau,\BT,1}$ are equidimensional of dimension $[K:\Qp]$.
  \item The special fibres $\cC^{\dd,\BT,1}$ and $\cC^{\tau,\BT,1}$ are reduced.
  \item If $\tau$ is non-scalar, then $\cC^{\tau,1}$  has $2^f$  irreducible components $\cC^{\tau}(J)$, in bijection with the set of profiles $J$.
\end{enumerate}
\end{thm}

\subsection{\'Etale \texorpdfstring{$\varphi$}{phi}-modules}

\begin{defn}
If $A$ is a $\Z/p^a\Z$-algebra for some $a \ge 1$, then a \emph{weak \'etale $\varphi$-module with $A$-coefficients} for $K'$ is a finitely generated  $\mathfrak{S}_A[1/u]$-module $M$ together with a semilinear morphism
\[  \varphi_{M}: M\rightarrow M    \]
such that the linearization
\[  \Phi_M: \varphi^\ast M\ := \mathfrak{S} \otimes_{\varphi,\mathfrak{S}} M \rightarrow M \]
is an isomorphism.  Weak \'etale modules with $A$-coefficients for $K$ are defined identically, but with $\fS^0_A[1/v]$ in place of $\fS_A[1/u]$.

A \emph{weak \'etale $\varphi$-module with $A$-coefficients and descent data from $K'$ to $K$} is a weak \'etale $\varphi$-module  with $A$-coefficients for $K'$ together with additional semilinear morphisms 
\[   \widehat{g}: M\rightarrow M,\qquad g\in \Gal(K'/K)   \]
such that each $\widehat{g}$ commutes with $\varphi_{M}$, and $\widehat{g}_1\circ\widehat{g}_2=\widehat{g_1g_2}$ for all $g_1, g_2\in\Gal(K'/K)$.

An \emph{\'etale $\varphi$-module} is a weak \'etale $\varphi$-module such that $M$ is furthermore projective as an $\fS_A[1/u]$-module (resp.\ as an $\fS^0_A[1/v]$-module, for \'etale $\varphi$-modules for $K$).
\end{defn}

\begin{defn}
    We define $\cR^{\dd}$ to be the  stack over $\Spf(\cO)$ which associates to any $\cO/\varpi^a$-algebra $A$ the groupoid of rank $2$ \'etale $\varphi$ modules with $A$-coefficients and descent data.
\end{defn}

We will also sometimes want to make use of \'etale $\varphi$-modules without descent data. Write $\cR_2$ for the $\Spf(\cO)$-stack of rank $2$ \'etale $\varphi$-modules for $K$ (without descent data). The functor $\cR_2 \to \cR^{\dd}$ sending $M \leadsto M \otimes_{\gS^0[1/v]} \gS[1/u]$ is an equivalence, with inverse given by taking $\Gal(K'/K)$-invariants;\ \emph{cf.}\  \cite[Cor.~2.3.21]{EGmoduli}.

If $\gM$ is a Breuil--Kisin module with $A$-coefficients and descent data, then evidently $\gM[1/u]$ is an \'etale $\varphi$-module with $A$ coefficients and descent data. Inverting $u$ thus defines morphisms 
\begin{equation}\label{eq:morphism-to-R}
\cC^{\dd,\BT} \to \cR^\dd, \qquad \cC^{\tau,\BT} \to \cR^\dd .
\end{equation}

\begin{defn}
We define $\cZ^{\dd}$ and $\cZ^{\tau}$ to be the scheme-theoretic images (in the sense of \cite{emertongeeproper}) of the morphisms given in  \eqref{eq:morphism-to-R}. We further define $\cZ^{\dd,1}$, $\cZ^{\tau,1}$, and $\cZ^{\tau}(J)$ for each profile $J$ to be the scheme-theoretic images of $\cC^{\dd,\BT,1}$  $\cC^{\tau,\BT,1}$, $\cC^{\tau}(J)$  under the morphisms of \eqref{eq:morphism-to-R}. 
\end{defn}

Note that $\cZ^{\dd,1}$ and $\cZ^{\tau,1}$ are reduced as a consequence of Theorem~\ref{thm:stacks-results}(3), but need not be (and in general presumably will not be) the special fibers of $\cZ^{\dd}$ and~$\cZ^\tau$.

We recall some of the important properties of these stacks as established in \cite[Thm.~5.1.2, Prop.~5.1.4, Lem.~5.1.8, Prop.~5.2.20]{cegsB} and \cite[Thm~6.2]{cegsA}.

\begin{thm}\label{lem: C 1 and Z 1 are the underlying reduced substacks}

The stacks $\cZ^{\dd,1}$ and $\cZ^{\tau,1}$ are algebraic stacks of finite type over~$\cO$, while the stacks $\cZ^{\dd}$ and $\cZ^{\tau}$ are $\varpi$-adic formal algebraic stacks. Moreover:
\begin{enumerate}

  \item The  stacks $\cZ^{\dd,1}$ and  $\cZ^{\tau,1}$ are equidimensional of
dimension~$[K:\Qp]$.   %%% Prop 5.2.19 of cegsB 
  \item The stacks $\cZ^{\dd,1}$ and $\cZ^{\tau,1}$ are reduced.   %%% Lem 5.1.8 of cegsB
  \item The irreducible components $\overline{\cZ}(\sigma)$ of $\cZ^{\dd,1}$ are in bijection with non-Steinberg Serre weights $\sigma$. Furthermore, for each finite extension $\F'/\F$ the $\F'$-points of $\overline{\cZ}(\sigma)$ are precisely the Galois representations $\rhobar:G_K \to \GL_2(\F')$ having $\sigma$ as a Serre weight.
  
\item If $\tau$ is non-scalar, then $\cZ^{\tau}(J) = \overline{\cZ}(\sigmabar(\tau)_J)$ for each $J \in \cP_\tau$, and these are precisely the irreducible components of $\cZ^{\tau,1}$.
\end{enumerate}
\end{thm}

\subsection{Galois representations}\label{sec:galois-representations}
We fix a compatible sequence $\{\pi_n\}$ of $p^n$th roots of $\pi$;\ since $\gcd(e(K'/K),p)=1$, this determines a compatible sequence $\{\pi_n'\}$ of $p^n$th roots of $\pi'$ such that $(\pi_n')^{e(K'/K)}=\pi_n$.  Let $K_\infty:=\cup_n K(\pi_n)$ and let $K_\infty':=\cup_nK'(\pi_n')$.  Then we identify $\Gal(K'/K)$ and $\Gal(K_\infty'/K_\infty)$.

By Fontaine's theory of the field of norms, if $|A| < \infty$ then the category of weak \'etale $\varphi$-modules with $A$-coefficients is equivalent to the category of $A$-representations of $G_{K_\infty'}$. There are various ways to write down such a functor, and in particular we will need to compare the functors to Galois representations of \cite{cegsC} and \cite{gls12}. For this reason we now recall the explicit descriptions of these functors.

Let $\cO_{\cE}$ denote the $p$-adic completion of $\mathfrak{S}\left[\frac 1 u\right]$; it is a discrete valuation ring with uniformizer $p$ and residue field $k'(\!(u)\!)$.  We let $\cE$ denote the field of fractions of $\cO_{\cE}$.  Note that the actions of $\varphi$ and $\Gal(K'/K)$ extend naturally to $\cO_{\cE}$ and $\cE$.

Fix an algebraic closure $\overline K$ of $K$ with ring of integers $\cO_{\overline K}$, and an embedding $K_\infty'\hookrightarrow \overline{K}$, and set $R:=\varprojlim_{x\mapsto x^p}\cO_{\overline K}$.  Write $\underline\pi:=(\pi_n)_n, \underline\pi':=(\pi_n')\in R$ and write $[\underline\pi], [\underline\pi']\in W(R)$ for their multiplicative lifts in the Witt vectors $W(R)$.

We may define a $\varphi$-equivariant inclusion $\mathfrak{S}\hookrightarrow W(R)$ by sending $u\mapsto [\underline\pi']$, and this restricts to a $\varphi$-equivariant inclusion $\mathfrak{S}^0\hookrightarrow W(R)$ sending $v$ to $[\underline\pi]$.  This injection extends to inclusions
\[  \cO_{\cE}\hookrightarrow W(\mathrm{Frac}(R))   \]
and 
\[  \cE\hookrightarrow W(\mathrm{Frac}(R))\left[\tfrac 1 p\right].   \]
The residue field $\mathrm{Frac}(R)$ of $W(\mathrm{Frac}(R))$ contains a separable closure $k'(\!(u)\!)^{\sep}$ of the residue field of $\cO_{\cE}$, and this extension of residue fields corresponds to an unramified extension $\cE^{\nr}$ of $\cE$.  We let $\hatEnr$ denote its $p$-adic completion, and we let $\hatOEnr \subset \hatEnr$ denote its ($p$-adically complete) ring of integers.

\begin{defn}
    We define the covariant functor $T$ from weak \'etale $\varphi$-modules with $A$-coefficients to Galois representations, given by
\[  T(M):=\left(\hatOEnr \otimes_{\mathfrak{S}[1/u]}M\right)^{\varphi=1}.\]

If $M$ has descent data from $K'$ to $K$, this is actually a representation of $G_{K_\infty}$ because $G_{K_\infty}$ acts on both $\hatOEnr$ and $M$ (the latter through its quotient $\Gal(K_\infty'/K_\infty)\cong \Gal(K'/K)$).  
Equivalently, since $[K':K]$ is prime to $p$  we have that $M^0:=M^{\Gal(K'/K)}$ is an \'etale $\varphi$-module with $A$-coefficients for $K$, and the natural map
\[  \gS[1/u]\otimes_{\gS^0[1/v]}M^0\rightarrow M  \]
is an isomorphism. Write
\[ T_K(M^0) := \left(\hatOEnr \otimes_{\mathfrak{S}^0 [1/v]}M^0 \right)^{\varphi=1}\]
for the analogous functor on weak \'etale $\varphi$-modules with $A$ coefficients for $K$, without descent data. Then $T_K(M^0) \cong T(M)$ and the right-hand side becomes a representation of $G_{K_\infty}$ by transport of structure.
\end{defn}

\begin{defn}
There are further (contravariant) functors $T^\ast$ on weak \'etale $\varphi$-modules with $A$-coefficients (with descent data from $K'$ to $K$, or for $K$, respectively), defined as follows.
\begin{align*}
 T^\ast(M) & := \Hom_{\mathfrak{S}[1/u],\varphi}\left(M,\widehat{\cE^{\nr}}/\hatOEnr \right) \\
 T^\ast_K(M^0) & := \Hom_{\mathfrak{S}^0[1/v],\varphi}\left(M^0,\widehat{\cE^{\nr}}/\hatOEnr \right) .
 \end{align*}
 These are naturally an $A$-linear representation of~$G_{K_\infty}$, and when $M^0 = M^{\Gal(K'/K)}$ then evidently $T^\ast(M) \cong T^\ast_K(M^0)$. When pre-composed with the functor from Breuil--Kisin modules to \'etale $\varphi$-modules, $T_K^\ast$ is the functor used in \cite{gls12} (and denoted $T_{\mathfrak{S}}$ in \S3 of that reference).
    \end{defn}

If $A$ is a $\Z/p^a\Z$-algebra and $\mathfrak{M}$ is a Breuil--Kisin module with $A$-coefficients and descent data, so that $\mathfrak{M}[1/u]$ is an \'etale $\varphi$-module with $A$-coefficients and descent data, then we will freely write
\[  T(\mathfrak{M}):=T(\mathfrak{M}[1/u]),\qquad T^\ast(\mathfrak{M}):=T^\ast(\mathfrak{M}[1/u]) .   \]

%\textcolor{red}{BL: I think this is used above but we should either find a citation or move to appendix}
%
%Write $T'$ for the covariant functor from \'etale $\varphi$-modules with descent data to Galois representations as defined in \cite{cegsC}. In \cite{gls12}, a contravariant functor is defined for passage from \'etale $\varphi$-modules \textit{without} descent data to Galois representations, which we denote by $T$ here. For a Breuil--Kisin module $\gM$ with descent data and $R$ coefficients where $R$ is a finite field extension of $\F'$, write $\cM' = \gM[1/u]$ \textcolor{purple}{HL: why do we use $P,P'$ instead of $\cM,\cM'$? (Below we use the latter.)}\textcolor{brown}{KK: edited}. Then $\cM := \cM'^{\Gal(K'/K)}$ is an \'etale $\varphi$-module without descent data, and $\cM' \cong k'(\!(u)\!) \otimes_{k(\!(v)\!)} \cM$ (as \'etale $\varphi$-modules with descent data, with $G$ acting on the tensor product trivially on the right-hand factor and via $u \mapsto h(g) u$ on the left-hand factor). Let \[\cM^{\vee}_{R} := \Hom_{k(\!(v)\!) \otimes_{\Fp} R}(\cM, k(\!(v)\!) \otimes_{\Fp} R)\] be the dual \'etale-$\varphi$ module which is the mod $p$ analogue of $M^{\vee}_{E}$ defined in the proof of \cite[Prop.~3.4 (3)]{gls12}.

%{\color{brown} KK: I am editing the statement below from $T(\cM) \cong T'(\cM')^{\vee}$ to $T(\cM^{\vee}) \cong T'(\cM')$. I think it might be harder to show that $T(\cM^{\vee}) = T(\cM)^{\vee}$ and I think we don't need it. The only input is in Thm 3.4, which I have edited to be in terms of $T(\cM^{\vee})$ instead of $T(\cM)^{\vee}$.}

\begin{lemma}\label{lem:gal-rep-comparison} Suppose that $A$ is a finite extension of $\F_p$ and $M$ is an \'etale $\varphi$-module with $A$-coefficients and descent data, and set $M^0 = M^{\Gal(K'/K)}$.  Let 
\begin{align*}
M^\vee & :=  \Hom_{\gS_A[1/u]}(M,\gS_A[1/u])\\
(M^0)^\vee & :=  \Hom_{\gS^0_A[1/v]}(M^0,\gS^0_A[1/v])
\end{align*}
be the {\upshape (}$A$-linear{\upshape)} dual $\varphi$-modules with coefficients {\upshape(}and descent data, in the first case{\upshape)}.  Then we have a functorial isomorphism
\[ T(M) \cong T^\ast_K((M^0)^\vee) \]
as representations of $G_{K_\infty}$.  
\end{lemma}

\begin{remark}\label{rem:frob-dual}
    The Frobenius on $M^\vee$ is defined by the formula $(\varphi f)(\sum_{i} s_i \varphi(m_i)) = \sum_i s_i \varphi(f(m_i))$ for any $s_i \in \gS_A[1/u]$ and $m_i \in M$;\ and similarly for $(M^0)^\vee$.
\end{remark}

\begin{proof}
Since $M$ is a finite projective $A \otimes \cO_{\cE}=\gS_A[1/u]$-module, we have $M\cong (M^\vee)^\vee$ and $(M^0)^\vee \cong (M^\vee)^0$. By the  discussion in the preceding paragraphs, we can therefore  reduce to proving that \[ T(M^\vee) \cong T^\ast(M).\]  Furthermore, we have a natural isomorphism
\begin{align}\label{eq:take-phi-invts-of-this}
     \widehat{\cE^{\nr}}/\hatOEnr \otimes_{\cO_{\cE}}M^\vee\xrightarrow\sim \Hom_{\cO_{\cE}}(M,\widehat{\cE^{\nr}}/\hatOEnr)
\end{align} 
which is $\varphi$- and $G_{K_\infty}$-equivariant.  Since $M$ (and hence $M^\vee$) is $p$-torsion, the left side of \eqref{eq:take-phi-invts-of-this} is isomorphic to
\[  \hatOEnr \otimes_{\cO_{\cE}} M^\vee  \]
and its $\varphi$-invariant subspace is simply $T({M}^\vee)$.  On the other hand, the $\varphi$-invariants of the right side of \eqref{eq:take-phi-invts-of-this} are $T^\ast(M)$ by definition, so we are done.
\end{proof}

\section{Points of \texorpdfstring{$\cC^{\tau}(J)$}{C\^{}tau(J)} and \texorpdfstring{$\cZ^{\tau}(J)$}{Z\^{}tau(J)}}\label{sec:the-setup}
%\mar{DS: rename! this should be the name of sec 3.1, not the whole section}

In \cite{cegsC} the irreducible components of the stack $\cC^{\tau,\BT,1}$ are studied by analyzing a morphism from $\cC^{\tau,\BT}$ to a certain auxiliary stack $\mathcal{G}_\eta$ (the ``gauge stack''). Our goal in the first part of this section is to explain another approach to describing the components of $\cC^{\tau,\BT,1}$, in terms of the notion of \emph{shape}, and then to relate this description to the one from \cite{cegsC}. The advantage of our approach is that we are able to give a complete description of all of the points of each component of $\cC^{\tau,\BT,1}$, whereas  \cite{cegsC} only describes a dense set of points;\ see Corollary~\ref{cor:partial-frob-C(J)}. As an application, in Section~\ref{sec:finite-type-z} we are able to characterize the $\F'$-valued points of the stacks $\cZ^\tau(J)$ for each finite extension $\F'/\F$;\ see Theorem~\ref{prop:modified-basis-of-invariants}.

%There are two different ways of enumerating the irreducible components of $\cC^{\tau,\BT,1}$.  The first uses the concept of emph{shape}, which measures the positions of matrices for the relative Frobenius operators, and the second uses the concept of \emph{profile}.  In this section, we briefly recall both concepts, and we show they are compatible.

\subsection{Irreducible components via shape}\label{sec:shape}
%\mar{RB: We should decide whether our coefficients are $A$ or $R$ DS: They should certainly be $A$.}

Shapes for rank $2$ Breuil--Kisin modules with tame descent were introduced by Breuil in \cite{breuillatticconj}, further developed by \cite{CDMa} to study tamely Barsotti-Tate deformation rings (and called \emph{genre} there), and eventually generalized to higher dimensions in \cite{LLLM} and \cite{LLLM-models}. Each field-valued point of $\mathcal{C}^{\tau,\BT,1}$ has an associated shape  which describes the divisibility by $u$ of certain entries in the matrices of the partial Frobenius maps. 

% \textcolor{teal}{HW: to do intro to this section: after reordering. Section 2.4 to move higher up.}
  Throughout this section we fix a non-scalar tame inertial type $\tau = \eta \oplus \eta'$. We will need the following notation. For each~$i$, we let $k_i, k'_i \in [0, p^{f'} - 1)$ be such that  $\eta = (\widetilde{\omega}'_i)^{k_i}$ and $\eta' = (\widetilde{\omega}'_i)^{k'_i}$. 
Let $\gamma_i \in [0, p-1]$ be the unique integers such that
\[  \eta \eta^{\prime -1} = \prod_{i \in \Z/f'\Z} (\widetilde{\omega}'_i)^{\gamma_i}. \]
Note that we are implicitly considering $\eta$ and $\eta'$ as an \emph{ordered} pair of characters.  The formula
\begin{equation}\label{eq:useful-k-equation} p[k'_{i-1} - k_{i-1}] - [k'_i - k_i] = (p^{f'}-1)(p-1-\gamma_i)\end{equation}
is often useful, where for any $a \in \Z/f'\Z$, $[a]$ denotes the unique element of $[0, p^{f'}-1)$ such that $[a]$ is congruent to $a$ mod $p^{f'}-1$. For brevity we will shorten $[k_i - k'_i]$ and $[k'_i - k_i]$ to $\ell_i,\ell'_i$ respectively. Since the type $\tau$ is nonscalar, $\ell_i$ and $\ell'_i$ are always both nonzero, and we have $\ell_i + \ell'_i = p^{f'}-1$.

%The tame type $\tau = \eta \oplus \eta'$ will be fixed and non-scalar. \textcolor{red}{BL: cuspidal case to be added later}. %By twisting if necessary, we may assume without loss of generality that $\eta' =1$. %The constructions in this section will follow the ideas of \cite{CDMa} and \cite{LLLM}.

\begin{defn}\label{def:eigenbasis}
Suppose that $A$ is an $\F$-algebra and let $\gM$ be an object of $\cC^{\tau,\BT,1}(A)$. By Lemma~\ref{lem:zariski-locally-free} and the hypothesis that $I(K'/K)$ has order prime to $p$, Zariski locally on $A$ one can choose a  basis $\beta_i = ({\eb}_i, {\fb}_i)$ of $\gM_i$ such that $\Gal(K'/K)$ acts on ${\eb}_i,{\fb}_i$ via $\eta,\eta'$ respectively. Furthermore, in the cuspidal case we can and do  suppose that $\mathfrak{c}({\eb}_i) = {\fb}_{i+f}$ and $\mathfrak{c}({\fb}_i) = {\eb}_{i+f}$. As in \cite{LLLM}, we call $\beta = (\beta_i)_{i\in \Z/f'\Z}$ an \emph{eigenbasis} of $\gM$. 
\end{defn}

\begin{defn}
    If the \'etale $\varphi$-module $M$ is free and has basis $\beta = (\beta_i)$, then  \emph{the matrix of the partial Frobenius map} $\Phi_{M,i}$ with respect to the basis $\beta$ is the matrix of $\Phi_{M,i}$ with respect to the basis $(1\otimes \beta_{i-1})$ of $\varphi^* M_{i-1}$ and the basis $\beta_i$ of $M_i$.  We will also use the same terminology in the context of
    a Breuil--Kisin module $\gM$ and a basis for $\gM[1/u]$ (which need not be a basis for $\gM$).
\end{defn}

Suppose that $\gM$ has an eigenbasis $\beta$, and let $C_{\beta,i}$ denote the matrix of the partial Frobenius map $\Phi_{\gM,i}$ with respect to $\beta$.  Since $\Phi_{\gM,i}$ commutes with the descent data we find that 
\begin{equation}\label{eq:Cbeta} C_{\beta,i} = \begin{pmatrix*}[r]
a_i & u^{\ell'_i} b_i \\
u^{\ell_i} c_i & d_i
\end{pmatrix*}  
\end{equation}
for some $a_i, b_i, c_i, d_i \in A[\![v]\!]$, meaning that $\Phi_{\gM,i}(1 \otimes \eb_{i-1}) = a_i \eb_i + u^{\ell_i} c_i \fb_i$ and similarly for 
$\Phi_{\gM,i}(1 \otimes \fb_{i-1})$. In the cuspidal case we  additionally  compute that 
\begin{equation}\label{eq:cuspdial-C-matrix}
C_{\beta,i+f} = \Ad \begin{pmatrix}
0 & 1 \\
1 & 0
\end{pmatrix}(C_{\beta,i})
\end{equation}
where $\Ad A(B) = ABA^{-1}$.

Any change of basis from an eigenbasis $\beta$ to another eigenbasis $\beta'$ is similarly encoded by an $f'$-tuple of matrices of the form
\[U_i = \begin{pmatrix*}[r]
x_i & u^{\ell'_i} y_i\\
u^{\ell_i} z_i & w_i 
\end{pmatrix*}\] so that $\beta'_i = \beta_i \cdot U_i$. The matrix $U_i$ has  $x_i,y_i,z_i,w_i$ in  $A[\![v]\!]$ and determinant in $A[\![v]\!]^{\times}$,  and in the cuspidal case one again has 
\[ U_{i+f} = \Ad \begin{pmatrix}
0 & 1 \\
1 & 0
\end{pmatrix}(U_i).\]
Under this change of basis, the matrices $C_{\beta,i}$ change by the formula
\[C_{\beta',i} = U_i^{-1} \cdot C_{\beta,i} \cdot \varphi(U_{i-1}).
\] 
%  Conjugating by $\begin{pmatrix}
% u^{-t'_i} & 0 \\
% 0 & 1
% \end{pmatrix}$, we see that the set of eigenbases is a torsor under an Iwahori subgroup of $\GL_2(A[\![v]\!])$.  

Observe that the matrices $U_i$ are diagonal modulo $u$, so that the diagonal entries $x_i,w_i$ must be units in $A[\![v]\!]$. In particular whether or not the entries $a_i,d_i$ of $C_{\beta,i}$ are divisible by $v$ is unchanged under change of eigenbasis. We can therefore make the following definition.

%\textcolor{red}{TODO: combine this into one definition with parts.} \textcolor{brown}{An attempt is below.}

\begin{defn}
Suppose that $\gM$ has an eigenbasis. We define the \emph{shape of $\gM$ at $i$} to be
\begin{itemize}
\item $\txI_{\eta}$ if $v \mid a_i$ and $v \nmid d_i$,
\item $\txI_{\eta'}$ if $v \nmid a_i$ and $v \mid d_i$, and
\item $\txII$ if  $v \mid a_i, d_i$ both.
\end{itemize}
The argument in Lemma~\ref{lem:L-cover-C} below proves that if $\gM$ has an eigenbasis and $A$ is a domain, then $\gM$ has a shape;\ but in general $v$ may divide neither $a_i$ nor $d_i$, in which case the shape of $\gM$ at $i$ does not exist. The \emph{shape of $\gM$}, if it exists, is the $f'$-tuple whose $i$-th entry is the shape of $\gM$ at $i$. 

In the cuspidal case we observe from \eqref{eq:cuspdial-C-matrix} that the shape of $\gM$ at $i$ is $\txI_{\eta}$ if and only if the shape of $\gM$ at $i+f$ is $\txI_{\eta'}$, and the shape of $\gM$ at $i$ is $\txII$ if and only if the shape of $\gM$ at $i+f$ is also $\txII$.
\end{defn}

We remark that having shape $\txI_{\eta}$ or $\txII$ (Zariski locally) at $i$ is a closed condition, and similarly for having shape $\txI_{\eta'}$ or $\txII$  at $i$;\ and their intersection is the condition of having shape $\txII$ at $i$. Having shape $\txI_\eta$ (respectively $\txI_{\eta'}$) is not a closed condition, but it is locally closed. %\mar{HL: perhaps I'm too conservative, but I think it is confusing if we say that having certain shape is a closed condition after saying that shape does not exist in some cases. Do we want to say ``having shape ... at $i$ Zariski locally is a closed condition''?

\begin{defn} Let $J \subset \Z/f'\Z$ be a profile.
    We define $\cL^{\tau}(J) \subset \cC^{\tau,\BT,1}$ to be the substack such that an object $\gM \in \cC^{\tau,\BT,1}(A)$ lies in $\cL^{\tau}(J)(A)$ if and only if Zariski locally on $A$ the Breuil--Kisin module $\gM$ has shape $\txI_\eta$ or $\txII$ when $i \in J$, and shape $\txI_{\eta'}$ or $\txII$ when $i \not\in J$. The inclusion $\cL^{\tau}(J) \subset \cC^{\tau,\BT,1}$ is a closed immersion (since the condition of being a closed immersion is checkable locally on $A$).
\end{defn}

\begin{lemma}\label{lem:L-cover-C}
    We have $|\cC^{\tau,\BT,1}| = \bigcup_J |\cL^{\tau}(J)|$, the union taken over profiles $J$.
\end{lemma}

\begin{proof}
    Suppose that $A$ is a field and $\gM \in \cC^{\tau,\BT,1}(A)$. The Breuil--Kisin module $\gM$ has an eigenbasis $\beta$, and Proposition~\ref{prop:reduced-f-strong-det} implies $\det C_{\beta,i} \in v A[\![u]\!]^{\times}$ for each $i$. Comparing with \eqref{eq:Cbeta} and recalling that $u^{\ell_i} \cdot u^{\ell'_i} = v$ we find in particular that $v \mid a_i d_i$ as elements of $A[\![v]\!]$. We conclude that either $v \mid a_i$ or $v \mid d_i$. Therefore $\gM$ has a shape, and lies in $\cL^{\tau}(J)(A)$ for some $J$ (possibly more than one).
\end{proof}

In some computations (such as the one upcoming), rather than working with the matrices of $\Phi_{\gM,i}$ with respect to an eigenbasis, it is more convenient to write the matrices of  $\Phi_{\gM,i}$ in terms of bases for the $\eta'$-eigenspaces of $\varphi^*\gM_{i-1}$ and $\gM_i$. Concretely, if $\beta = (\beta_i)$ is an eigenbasis with $\beta_i = (\eb_i,\fb_i)$, then the $\eta'$-eigenspace of $\gM_i$ has a basis given by $(u^{\ell'_i}{\eb}_i, {\fb}_i )$, while that of $\varphi^{*} \gM_{i-1}$ has a basis given by $( u^{\ell'_i} \otimes {\eb}_{i-1}, 1 \otimes {\fb}_{i-1})$. The matrix of $\Phi_{\gM,i}$ in terms of these bases is 
 \begin{equation}\label{eq:def-of-Abeta}  A_{\beta,i} := \Ad \begin{pmatrix}
u^{-\ell'_i} & 0 \\
0 & 1
\end{pmatrix} \left( C_{\beta,i}\right) = \begin{pmatrix*}[r]
a_i & b_i \\
v c_i & d_i
\end{pmatrix*}  . \end{equation}
Note that this is not the same as ``the matrix with respect to the basis $(u^{\ell'_i}e_i,f_i)_{i}$'', since $1\otimes u^{\ell'_{i-1}}e_i \neq u^{\ell'_i}\otimes e_{i-1}$ in general.
These 
matrices have entries in $A[\![v]\!]$; for that reason, this process is sometimes called ``removing the descent data.'' In the cuspidal case, one checks using $\ell'_i + \ell'_{i+f} = p^{f'}-1$ that
 \begin{equation}\label{eq:A-cuspidal}  A_{\beta,i+f} := \Ad \begin{pmatrix}
0 & 1 \\
v & 0
\end{pmatrix} A_{\beta,i} = \begin{pmatrix*}[r]
d_i & c_i \\
v b_i & a_i
\end{pmatrix*}  .
\end{equation}

%it is convenient to work with
%where the notation $\text{Ad}\; A \; (B)$ means $A B A^{-1}$. %Then $A_{\beta,i}$ is the Frobenius matrix with respect to %these bases, which we will call the $i$-th partial Frobenius. %When $\tau$ is cuspidal, since $\mathfrak{c}({\eb}_i) = %{\fb}_{i+f}$ and $\mathfrak{c}({\fb}_i) = {\eb}_{i+f}$, $C_{\beta,i}$ (resp. $A_{\beta,i}$) completely determines $C_{\beta,i+f}$ (resp. $A_{\beta,i}$) for each $i$.

Setting $I_i = \Ad\begin{pmatrix}
u^{-\ell'_i} & 0 \\
0 & 1
\end{pmatrix} \left(U_i \right) = \begin{pmatrix*}[r]
x_i & y_i \\
vz_i & w_i
\end{pmatrix*}$, we get the change of basis formula
\begin{equation}\label{eq:cofb}
    A_{\beta',i} = I_{i}^{-1} A_{\beta,i} \; \text{Ad} \begin{pmatrix}
    v^{p-1-\gamma_{i}} & 0 \\
    0 & 1
    \end{pmatrix} (\varphi(I_{i-1})\!).
\end{equation}
using \eqref{eq:useful-k-equation}. In the cuspidal case the matrices $I_i$ and $I_{i+f}$ are related by the formula
\begin{equation}\label{eq:I-cuspidal}  I_{i+f} := \Ad \begin{pmatrix}
0 & 1 \\
v & 0
\end{pmatrix} I_{i} = \begin{pmatrix*}[r]
w_i & z_i \\
v y_i & x_i
\end{pmatrix*}  .
\end{equation}

%We maintain our condition on $J$ going forwards: When $\tau$ is a principal series type, $J \subset \bZ/f'\bZ$ is arbitrary, but in the cuspidal case, we require that $i \in J$ if and only if $i+f \not\in J$.

Let $L^+ \GL_2$ denote the positive loop group over $\F$ with variable $v$, i.e., the group ind-scheme whose $A$-points are $L^+ \GL_2(A) = \GL_2(A[\![v]\!])$. Write $\cC^{\tau,1}$ for the $\F$-stack of rank $2$ Breuil--Kisin modules with coefficients and descent data, of type $\tau$, and of height at most $1$;\ in other words, all of the defining conditions of $\cC^{\tau,\BT,1}$ except the strong determinant condition. There is a map
\begin{equation}\label{eq:cover-unfactored}
 (L^+ \GL_2)^{f} \rightarrow \cC^{\tau,1}
\end{equation}
given by sending the $f$-tuple of matrices $B_i \in \GL_2(A)$ for $0 \le i < f$ to the Breuil--Kisin module $\gM$ with 
$\gM_i = A[\![u]\!]{\eb}_i \oplus A[\![u]\!]{\fb}_i$ for $i \in \Z/f'\Z$,  an eigenbasis $\beta$ with $\beta_i = (\eb_i,\fb_i)$ for $i \in \Z/f'\Z$, and partial Frobenius maps $\Phi_{\gM,i}$ having matrices $A_{\beta,i}$ for $0 \le i < f$ defined by the formulas
\begin{equation}\label{eq:a-beta-i-map} A_{\beta,i} = 
\begin{dcases}
B_i \begin{pmatrix}
v & 0 \\
0 & 1
\end{pmatrix} &\text{ if } i \in J\\
\begin{pmatrix}
1 & 0\\
0 & v
\end{pmatrix} B_i &\text{ if } i \not\in J.
\end{dcases}
\end{equation}
In the cuspidal case this means that the matrices $A_{\beta,i+f}$ for $0 \le i < f$ are determined in terms of $A_{\beta,i}$ by the formula \eqref{eq:A-cuspidal}.
%\mar{KK: I think we can upgrade this directly to a proof that $\cL^{\tau}(J)$ is reduced without straightening: Every BK module admits an eigenbasis Zariski locally. Therefore, (locally,) every map to $\cL^{\tau}(J)$ factors through $(L^+\GL_2)^{f}$ and so through $\cL^{\tau}(J)_{\text{red}}$, showing that $\cL^{\tau}(J)$ is reduced. More pedantically, 
%\textbf{1)} The map $\pi$ is surjective on points valued in local rings. Copying an argument in GKKSW, for $\gM$ defined over a local ring $A$, $\gM_i/u\gM_i = \eta \oplus \eta'$. If $\tilde{x}_i$ and $\tilde{y}_i$ are lifts of the $\eta$ and $\eta'$ eigenvectors of $\gM_i/u\gM_i$, then an eigenbasis for $\gM_i$ is given by $e_i := \frac{1}{e}\sum\limits_{j = 0}^{e - 1} g^{i}(\tilde{x}_i)\eta(g^{-i})$ and
%$f_i := \frac{1}{e}\sum\limits_{j = 0}^{e - 1} g^{i}(\tilde{y}_i)\eta^{\prime}(g^{-i})$, where $g$ is a generator of $I(K'/K)$. \textbf{2)} $\cL^{\tau}(J)$ is an algebraic stack of being a closed substack of one, and admits a smooth cover $U$ from a scheme. Let $A$ be the stalk at a point of $U$ - thus there exists a smooth map from $\Spec A$ to $\cL^{\tau}(J)$. The map factors through $L^{+}\GL_2^f$ by (1).
%\textbf{3)} As $L^{+}\GL_2$ is reduced, $\pi$ factors through $\cL^{\tau}(J)_{\text{red}}$, and so does the map from $\Spec A$. Since the pullback along $\cL^{\tau}(J)_{\text{red}} \to \cL^{\tau}(J)$ induces id on $A$, $A$ is reduced (because $A$ maps smoothly to $\cL^{\tau}(J)_{\text{red}}$), thus so is $U$ and $\cL^{\tau}(J)$.}

\begin{lemma}
    The map of \eqref{eq:cover-unfactored} factors through the closed immersion $\cL^{\tau}(J) \subset \cC^{\tau,1}$ to  give a surjective map
    \[ \pi :  (L^+ \GL_2)^{f} \rightarrow \cL^{\tau}(J).\]
    Furthermore $\cL^{\tau}(J)$ is reduced and irreducible.
\end{lemma}
\begin{proof}
Since the loop group $L^+ \GL_2$ is reduced and the determinant of each matrix $A_{\beta,i}$ in \eqref{eq:a-beta-i-map} lies in $v A[\![v[\!]^\times$, it follows from Proposition~\ref{prop:reduced-f-strong-det} that the map \eqref{eq:cover-unfactored} factors through the closed immersion $\cC^{\tau,\BT,1} \subset \cC^{\tau,1}$. 
%\textcolor{pink}{Maybe we can add a line explaining why the positive loop group is reduced(since the points form an open sub-scheme of a reduced scheme).}
Since by construction the upper-left entry of $A_{\beta,i}$ is divisible by $v$ if $i \in J$, and similarly for the bottom-right entry if $i \not\in J$, the image in fact lies in $\cL^{\tau}(J)$. We therefore obtain the claimed factorization $\pi$.  In fact
since $L^+\GL_2$ is reduced,  by 
\cite[\href{https://stacks.math.columbia.edu/tag/050B}{Tag 050B}]{stacks-project} the map $\pi$ must factor through the closed immersion $\cL^{\tau}(J)_{\red} \subset \cL^{\tau}(J)$;\ here $\cL^{\tau}(J)_{\red}$ denotes the underlying reduced substack of $\cL^{\tau}(J)$.

Let $A$ be a ring and suppose that $\gM \in \cL^\tau(J)(A)$ admits an eigenbasis $\beta$. For each $i \in J$ we have $v \mid a_i$, say $a_i = v a'_i$, and $A_{\beta,i}$ has the factorization
\[ A_{\beta,i} = B_i \begin{pmatrix}
v & 0 \\
0 & 1
\end{pmatrix}
\] with $B_i = \begin{pmatrix}
a'_i & b_i \\
c_i & d_i
\end{pmatrix}$;\ and analogously when $i \not\in J$. It follows that   $\gM$  is in the image of $\pi$.   When $A$ is a field (or indeed any local ring), every $\gM \in \cL^{\tau}(J)(A)$ admits an eigenbasis, and therefore is in the image of $\pi$;\ this establishes the surjectivity of $\pi$. The irreducibility of $\cL^{\tau}(J)$ now follows  from the surjectivity of $\pi$ and the irreducibility of $(L^+\GL_2)^f$. 

%The formal smoothness of $\pi$ is straightforward, as one can always lift any eigenbasis of a Breuil--Kisin module with respect to infinitesimal thickenings.  Indeed, the space of eigenbases is an ind-torsor under an ind-algebraic group whose successive quotients are reduced (and hence smooth).\mar{DS: am I happy with this sentence? RB: I've added a (hopefully) clarifying sentence}

Now  consider any morphism $X \to \cL^{\tau}(J)$ with $X$ a scheme over $\F$. Cover~$X$ by affine opens $\Spec A$ such that the induced object of $\cL^{\tau}(J)(A)$ admits an eigenbasis. 
% \mar{KK: Do we need to say anything about the fact that an eigenbasis always exists Zariski locally? DS: I think this is already said in Def 3.2.}
By the argument in the previous paragraph,  each of the maps $\Spec A \to \cL^\tau(J)$ lifts through $\pi$, and therefore factors through $\cL^{\tau}(J)_{\red}\subset \cL^{\tau}(J)$. It follows that the morphism $X \to \cL^{\tau}(J)$ itself factors through $\cL^{\tau}(J)_{\red}$. Thus $\cL^{\tau}(J)(X) = \cL^{\tau}(J)_{\red}(X)$ for all schemes $X$;\ it follows that $\cL^{\tau}(J) = \cL^{\tau}(J)_{\red}$ and therefore $\cL^\tau(J)$ is reduced.
\end{proof}

\begin{thm}\label{thm:Ltau=Ctau}
    We have $\cL^\tau(J) = \cC^{\tau}(J)$ for all profiles $J$.
\end{thm}

\begin{proof}
We have just proved that the stacks $\cL^{\tau}(J)$ are reduced and irreducible, and by Lemma~\ref{lem:L-cover-C} their union is $\cC^{\tau,\BT,1}$. On the other hand, there are $2^f$ stacks $\cL^{\tau}(J)$, and $\cC^{\tau,\BT,1}$ has $2^f$ irreducible components. It follows that each $\cL^{\tau}(J)$ must be a different one of the irreducible components.

By \cite[Prop.~5.4.2]{cegsC}, the union $\cup_{J \ni i} \,\cC^{\tau}(J)$ (the union over profiles $J$ containing $i$) is the zero locus in $\cC^{\tau,\BT,1}$ of the $\eta$-isotypic part of $\overline{\Phi}_{\gM,i} : \varphi^* (\gM/u\gM)_{i-1} \to (\gM/u\gM)_i$. In the eigenbasis $\beta$, this map is multiplication by $a_i$ (mod $u$). Since $a_i \in A[\![v]\!]$, we have $u \mid a_i$ if and only if $v \mid a_i$. Therefore $\cup_{J \ni i} \,\cC^{\tau}(J) = \cup_{J \ni i} \,\cL^{\tau}(J)$. Since this equality holds for each $i$ we deduce that $\cL^{\tau}(J) = \cC^{\tau}(J)$.
\end{proof}

\begin{cor}\label{cor:partial-frob-C(J)}
If $A$ is an $\F$-algebra, then Zariski locally on $A$, $\cC^{\tau}(J)(A)$ is precisely the groupoid of Breuil--Kisin modules with partial Frobenius matrices $A_{\beta,i}$ of the form
\begin{align*}
B_i \begin{pmatrix} v & 0 \\
0 & 1 \end{pmatrix} \text{ if } i \in J, \: \text{ and } \:
\begin{pmatrix} 1 & 0 \\
0 & v \end{pmatrix} B_i \text{ if } i \not\in J,
\end{align*}
for some $B_i \in \GL_2(A[\![v]\!])$. 
\end{cor}

\subsection{Finite type points of \texorpdfstring{$\cZ^\tau(J)$}{Z\^{}tau(J)}}\label{sec:finite-type-z}
As an application of Corollary~\ref{cor:partial-frob-C(J)} 
we are able to give a direct description  of the finite type points of $\cZ^{\tau}(J)$, by which we mean the points of $\cZ^{\tau}(J)(\F')$ for each finite extension $\F'/\F$.  
Recall from Section~\ref{sec:ill-sw} that to the pair $\tau$ and $J$ we have  associated a tuple of integers $s_{J,i}$ and a character~$\Theta_J$, which is identified via the Artin map with the character $\Theta_J : k^\times \to \F^\times$ such that
\begin{equation} \Theta_J \circ \mathrm{N}_{k'/k} = \eta' \otimes \prod_{i \in \Z/f'Z} (\kappa'_i)^{t_{J, i}};\end{equation}
see equations \eqref{def:s_J} and \eqref{def:theta_J}. Write
\[ \Theta_J = \prod_{i=0}^{f-1} \omega_i^{\theta_{J,i}} \] for some integers $\theta_{J,i}$. In the cuspidal case by definition we have $\theta_{J,i+f} = \theta_{J,i}$ for all $i$. The main result of this section is the following.

\begin{thm}\label{prop:modified-basis-of-invariants} Suppose $\F'/\F$ is a finite extension and let $M$ be an \'etale-$\varphi$ module 
with $\F'$-coefficients and descent data from $K'$ to $K$. Write $M^0$ for the $\Gal(K'/K)$-invariants of $M$. Then $M$
is a finite type point of $\cZ^{\tau}(J)$ if and only if $(M^0)_i$ has a basis $x_i,y_i$ for each $i \in \Z/f\Z$ such that the partial Frobenius maps $\Phi_{M^0,i}$  with respect to the basis $(x_i,y_i)_{i \in \Z/f\Z}$
have matrices
\[
    B_{i} \begin{pmatrix}
    v  & 0 \\
    0 & v^{-s_{J, i}}
    \end{pmatrix} v^{-\theta_{J,i}}
\]
 for matrices $B_{i} \in \GL_2(\F'[\![v]\!])$.
\end{thm}

\begin{remark}\label{rem:ordering-irrelevant}
 In the statement of the proposition, the ordering of the diagonal elements of the diagonal matrix is irrelevant, in the sense that they can be swapped by a suitable change of basis. Specifically, if $J' \subset \Z/f\Z$ is any subset, then replacing $(x_{i-1},y_{i-1})$ with $(y_{i-1},x_{i-1})$ for each $i \in J'$ swaps the order of the diagonal elements, 
at the cost of multiplying $B_{i-1}$ and $B_i$ on the left and right by $\left(\begin{smallmatrix} 0 & 1 \\ 1 & 0 \end{smallmatrix}\right)$ respectively for each $i \in J'$.

\end{remark}
\begin{proof}[Proof of Theorem~\ref{prop:modified-basis-of-invariants}]
The stack $\cZ^{\tau}(J)$ is the scheme-theoretic image of the morphism $\cC^\tau(J) \to \cR^{\dd,1}$;\ this morphism is proper because it is the composition of the closed immersion $\cC^\tau(J) \to \cC^{\tau,\BT,1}$ with the proper morphism $\cC^{\tau,\BT,1} \to \cR^{\dd,1}$. It follows from \cite[Lem~3.2.14]{emertongeeproper} that the $\F'$-points of $\cZ^{\tau}(J)$ are precisely the $\F'$-points of $\cR^{\dd,1}$ whose fiber in $\cC^\tau(J) \to \cR^{\dd,1}$ is nonempty;\ in other words, which have a eigenbasis $\beta =  (\beta_i)_i$ such that the partial Frobenius maps $\Phi_{M,i}$ have exactly the form given in Corollary~\ref{cor:partial-frob-C(J)}, i.e., such that the 
matrices $A_{\beta,i}$ are
\begin{align*}
B'_i \begin{pmatrix} v & 0 \\
0 & 1 \end{pmatrix} \text{ if } i \in J, \: \text{ and } \:
\begin{pmatrix} 1 & 0 \\
0 & v \end{pmatrix} B'_i \text{ if } i \not\in J.
\end{align*}
for some $B'_i \in \GL_2(\F'[\![v]\!])$. Note that these matrices can be written in a single formula as
\begin{equation}\label{eq:uniform-B}
A_{\beta,i} = \begin{pmatrix} 1 & 0 \\
0 & v^{\delta_{J^c}(i)} \end{pmatrix} 
B'_i \begin{pmatrix} v^{\delta_J(i)} & 0 \\
0 & 1 \end{pmatrix}.
\end{equation}

Before continuing the proof we need to introduce some additional notation. Write $\eta' = \prod_{i=0}^{f'-1} (\widetilde{\omega}'_i)^{\mu_i}$ for integers $\mu_i \in [0,p-1]$.  Then
$k'_i = \sum_{j=0}^{f'-1} p^{f'-1-j} \mu_{i+j+1}$ for all~$i$, and 
\begin{equation}\label{eq:get-mus}
    pk'_{i-1} - k'_i = (p^{f'}-1) \mu_i.
\end{equation} Next, the equality 
\[ \Theta_J = \prod_{i=0}^{f'-1} (\omega'_i)^{\theta_{J,i}} = \prod_{i=0}^{f'-1} (\omega'_i)^{\mu_i + t_{J,i}} \] 
implies the existence of integers $\nu_i$ such that
\begin{equation}\label{eq:def-of-nu}
    \theta_{J,i} = \mu_i + t_{J,i} + \nu_i - p \nu_{i-1}
\end{equation}
for all $i$. Now for each $i \in \Z/f'\Z$ we define a basis $(m_i,n_i)$ of the inertial invariants $M^{I(K'/K)}$ by the formula
\[
(m_i \ n_i ) = (e_i  \ f_i) \begin{pmatrix}
    u^{\ell'_i} & 0 \\ 0 & v^{\delta_{J^c}(i)} 
\end{pmatrix} u^{-k'_i} v^{-1+\nu_i}.
\]
Then one computes directly using equations \eqref{eq:def-of-Abeta}, \eqref{eq:uniform-B}, \eqref{eq:useful-k-equation}, \eqref{eq:get-mus}, and \eqref{eq:def-of-nu} 
that the matrix of $\Phi_{M,i}$ with respect to  the basis $(m_i,n_i)_{i \in \Z/f'\Z}$ is 
\[ B'_i \begin{pmatrix}
v^{\delta_J(i) - \gamma_i + t_{J_i}} & 0 \\
0 & v^{p \delta_{J^{c}}(i-1) - (p-1) + t_{J,i}}
\end{pmatrix} v^{-\theta_{J,i}}. \]
Substituting the definition of $t_{J,i}$, we obtain
\[ B'_i \begin{pmatrix}
v & 0 \\ 0 & v^{-s_{J,i}}
\end{pmatrix} v^{-\theta_{J,i}} \quad \text{ if }  i-1 \in J,  \qquad B'_i \begin{pmatrix}
v^{-s_{J,i}} & 0 \\ 0 & v
\end{pmatrix} v^{-\theta_{J,i}} \quad \text{ if } i-1 \not\in J.
\]
Finally, defining $(x_{i-1}, y_{i-1}) = (m_{i-1}, n_{i-1})$ if $i-1 \in J$ and  $(x_{i-1}, y_{i-1}) = (n_{i-1}, m_{i-1})$ if $i-1 \not\in J$, then by Remark~\ref{rem:ordering-irrelevant} the matrix of $\Phi_{M,i}$ with respect to the basis $(x_i,y_i)_{i \in \Z/f'\Z}$ is
 \[ B_i \begin{pmatrix}
v & 0 \\ 0 & v^{-s_{J,i}}
\end{pmatrix} v^{-\theta_{J,i}} \]
for suitable $B_i$.
%namely $B_i = \begin{pmatrix}  0 & 1 \\ 1 & 0 \end{pmatrix}^{\delta_{J^c}(i)} \cdot B'_i \cdot \begin{pmatrix}  0 & 1 \\ 1 & 0 \end{pmatrix}^{\delta_{J^c}(i-1)}$.

In the principal series case, where $I(K'/K) = \Gal(K'/K)$, this completes the proof of the `only if' direction of the proposition, and the `if' direction follows on observing that  our change of basis from the eigenbasis $(e_i,f_i)_i$ of $M$ to the  basis $(x_i,y_i)_i$ of $M^0$ is reversible and depends only on the initial data of $\tau$ and $J$. 

For the remainder of the proof, we assume that $\tau$ is cuspidal. To compute $M^0$ from $M^{I(K'/K)}$ we take further invariants under the element $\mathfrak{c} \in \Gal(K'/K)$ that fixes $\pi'$ and lifts the Frobenius element in $\Gal(K'/K)/I(K'/K) = \Gal(k'/k)$.  Another direct computation, recalling that $\mathfrak{c}(e_i,f_i) = (f_{i+f},e_{i+f})$, that $k_i = k'_{i+f}$, and that $\delta_J(i) = \delta_{J^c}(i+f)$, shows that
\begin{equation}\label{eq:c-basis-formula} \mathfrak{c}(m_i \ n_i) = v^{\xi_i} (n_{i+f} \ m_{i+f} )
\end{equation}
where for each $i$, we have
\[ \xi_i =\tfrac{\ell'_i + k_i - k'_i}{p^{f'}-1} + \nu_i - \nu_{i+f} - \delta_J(i). \] 
But using \eqref{eq:useful-k-equation}, \eqref{eq:get-mus} and \eqref{eq:def-of-nu}, %\mar{KK: the $p\xi_{i-1} - \xi_i$ computation uses is that $\gamma_{i+f} = p-1-\gamma_i$.  DS: I didn't think I used this? In any case I think it's OK not to recall}
we can compute
\begin{multline*}
p\xi_{i-1} - \xi_i =  (p-1-\gamma_i) - \mu_i + \mu_{i+f} + (\mu_i + t_{J,i} - \theta_{J,i}) \\ - (\mu_{i+f} + t_{J,i+f} - \theta_{J,i+f})  - p \delta_J(i-1) + \delta_J(i).
\end{multline*}
Since $\theta_{J,i} = \theta_{J,i+f}$, this simplifies to 
\[ p\xi_{i-1} - \xi_i =  t_{J,i} - t_{J,i+f} + (p-1-\gamma_i) - p\delta_J(i-1) + \delta_J(i) \]
and substituting the definition of $t_{J,i}$ (and recalling that $i-1 \in J$ if and only if $i+f-1 \not\in J$), we conclude finally that
\[ p \xi_{i-1} - \xi_i = 0 \]
for all $i$. But then iteratively we have $\xi_i = p^{f'} \xi_i$ for all $i$, and so in fact $\xi_i = 0$, and \eqref{eq:c-basis-formula} simplifies to \begin{equation}\label{eq:c-better-basis-formula} \mathfrak{c}(m_i \ n_i) =  (n_{i+f} \ m_{i+f} )
\end{equation} for all $i$. Therefore  $(m_i + n_{i+f}, m_{i+f} + n_i)$ are  $\Gal(K'/K)$-invariant, and are a basis for $M^0_i$. Furthermore the commutation relation between $\varphi$ and $\mathfrak{c}$  implies that if $X_i$ is the matrix of $\Phi_{M,i}$ with respect to the basis $(m_i,n_i)_{i \in \Z/f'\Z}$ then $X_{i+f} = \Ad \begin{pmatrix} 0 & 1 \\ 1 & 0 \end{pmatrix} (X_i)$. 

Now we can conclude just as in the principal series case:\ for each integer $0 \le i < f$ we define $(x_i,y_i) = (m_i+n_{i+f},m_{i+f}+n_i)$ if $i - 1\in J$ and $(x_i,y_i) = (m_{i+f}+n_{i},m_{i}+n_{i+f})$ if $i - 1 \not\in J$, and then the matrix of $\Phi_{M^0_i}$ with respect to  $(x_i,y_i)_{i = 0,\ldots,f-1}$ is
 \[ B_i \begin{pmatrix}
v & 0 \\ 0 & v^{-s_{J,i}}
\end{pmatrix} v^{-\theta_{J,i}} \]
for suitable $B_i$.
\end{proof}

\begin{remark}
     Readers who are familiar with the article \cite{gls12} will recognize that Theorem~\ref{prop:modified-basis-of-invariants} has a reinterpretation in terms of crystalline lifts. This will be discussed in the following section at Theorem~\ref{thm:gls-irregular}, Remark~\ref{rem:converse-remark-regular}, and Corollary~\ref{cor:gls-converse}.
\end{remark}

\section{Irregular loci}\label{sec:irregular}

We are now ready to prove many of our main results about irregular loci in the Emerton--Gee stacks. We begin in Section~\ref{sec:comparisons} by proving the isomorphism between the stacks $\cX^{\tau,\BT}$ and $\cZ^{\tau}$. This establishes the existence of substacks $\cZ^{\underline{r}}_{\red}$ of $\cZ^{\dd}$ whose finite type points are Galois representations having crystalline lifts of Hodge type $\underline{r}$, whenever $\underline{r}$ is $p$-bounded and irregular (\emph{cf.}\ Definition~\ref{def:hodge-type-names}). There is a Hodge type associated to each pair $(\tau,J)$, and in Section~\ref{sec:combinatorial-results} we prove the combinatorial fact that each $p$-bounded and irregular Hodge type arises from such a pair.  In Section~\ref{sec:comparison-section} we use this combinatorial input along with the results of Section~\ref{sec:shape} to prove that the stacks $\cZ^{\underline{r}}_{\red}$ are equal to the stacks $\cZ^{\tau}(J)$ for suitable choices of $\tau$ and $J$.

% \mar{DS: return to these paragraphs} In this section, we will show that when $J \not\in \cP_{\tau}$, $\cZ^{\tau}(J)$ can be interpreted as a locus of mod $p$ representations with crystalline lifts with certain irregular Hodge--Tate weights.

% Let $\alpha_i \in [0, p-1]$ be such that \textcolor{teal}{HW: $\alpha_i$ means something different in cegsC (need this in Sectio\textbf{}n 5)}
% \[  \eta' = \prod_i \omega_i^{\alpha_i} \]
% Thus, $pk'_{i-1} - k'_i = (p^{f'}-1)\alpha_i$. 

\subsection{A comparison of  \texorpdfstring{\cite{cegsB}}{[CEGSb]} and \texorpdfstring{\cite{EGmoduli}}{[EG2]} stacks}
\label{sec:comparisons}
%\mar{RB: I suggest moving the definition of $\cX_2$ (and its substacks) to a new subsection of section 2, since that's where we define the other stacks we need}

In this subsection only, we allow $K/\Qp$ to be an arbitrary finite extension (i.e., not necessarily unramified). Our goal in this subsection is to establish that the stack $\cZ^{\tau}$ is isomorphic to the Emerton--Gee stack $\cX^{\tau,\BT}$ of potentially Barsotti--Tate representations of type $\tau$.  For this, we need to begin with a few recollections from \cite{EGmoduli}.

Let $\cX_{2}$ be the Emerton--Gee stack of rank $2$ \'etale $(\varphi, \Gamma)$-modules over $K$ with $\cO$-algebra coefficients, i.e., the $d=2$ case of the stack $\cX_{d}$ discussed in the introduction. 
As before we write $\cR_2$ for the $\Spf(\cO)$-stack of rank $2$ \'etale $\varphi$-modules for $K$ (without descent data). We recall from \cite[Thm.~3.7.2]{EGmoduli} that there is a morphism 
\[ f: \cX_{2} \rightarrow \cR_2 \]
such that the map $\cX_2(A) \to \cR_2(A)$, for each complete local Noetherian $\cO$-algebra $A$, is given by restriction  from $G_K$ to $G_{K_\infty}$ on the corresponding Galois representations. We emphasize that despite the notation, the map $f$ is not simply ``forgetting $\Gamma$'', because the $(\varphi,\Gamma)$-modules of \cite{EGmoduli} are cyclotomic, whereas the \'etale $\varphi$-modules of $\cR_2$ are Kummer (following the terminology of \cite[Examples~2.1.2--2.1.3]{EGmoduli}).

Emerton and Gee construct substacks of $\cX_2$ which may be regarded as stacks of potentially crystalline representations with specified inertial and Hodge types. For our purposes we make the following definition.
%\mar{KK: Proposal to remove the line below r. DS: I'd like to keep it for consistency with EG notation in this section. But, when I eventually get to it, I suggest we rename $\cX(\underline{r}(\tau,J)$ as $\cX^\tau(J)$ and that will be what is mostly used in the paper.}
\begin{defn}\label{def:hodge type}
A \emph{Hodge type} of rank $d$ is a tuple of integers $ \underline{r} = \{r_{\kappa,j}\}_{\kappa: K \into E, 1 \le j \le d}$ with $r_{\kappa,1} \ge \cdots \ge r_{\kappa,d}$ for all $i$. 
% $r = \{r_{1, \bar{i}}, r_{2, \bar{i}}\}_{\bar{i} \in \Z/f\Z}$, we let $\Delta(r)_{\bar{i}} := |{r_{1, {\bar{i}}} - r_{2, {\bar{i}}}}|$.\mar{DS: reminder to myself to introduce ``tuple of Hodge--Tate weights'' and ``Hodge type''} We will call $r$ the labelled Hodge--Tate weights, and will use the following terms to describe it:
\end{defn}

The integers $r_{\kappa,i}$ should be thought of as being the $\kappa$-labeled Hodge--Tate weights of a $d$-dimensional representation of $G_K$. 
When $K/\Qp$ is unramified and we have indexed the embeddings $\kappa_i : K \into E$ by $i \in \Z/f\Z$, we will generally write $r_{i,j}$ in place of $r_{\kappa_i,j}$.

For each inertial type $\tau$  and Hodge type $\underline{r}$ of rank $2$, Theorem 4.8.12 of \cite{EGmoduli} guarantees the existence of a closed substack $\cX_2^{\crys,\tau,\underline{r}}$ of $\cX_2$ that is a $p$-adic formal algebraic stack, flat over $\cO$, such that $\cX_2^{\crys,\tau,\underline{r}}(A)$  for each finite flat $\cO$-algebra $A$ is the subgroupoid of potentially crystalline $G_K$-representations having inertial type $\tau$ and Hodge--Tate weights~$\underline{r}$.

Let $\BT$ denote the Hodge type $\underline{r}$ with $(r_{\kappa,1},r_{\kappa,2}) = (1,0)$ for all $\kappa$, and let $\mathrm{triv}$ denote the trivial inertial type. For brevity will write $\cX^{\tau,\BT}$  and $\cX^{\underline{r}}$  in place of $\cX_2^{\crys,\tau,\BT}$ and $\cX_2^{\crys,\mathrm{triv},\underline{r}}$ respectively, and we will write $\cX^{\tau,\BT}_{\red}$ and $\cX^{\underline{r}}_{\red}$ for their underlying reduced substacks.  By \cite[Thm.~4.8.12]{EGmoduli}, the finite type points of $\cX^{\underline{r}}_{\red}$ are precisely the mod $p$ representations with crystalline lifts of Hodge--Tate weights~$\underline{r}$.

By \cite[Thm.~4.8.14]{EGmoduli} the stack $\cX^{\tau,\BT}_{\red}$ is a union of irreducible components of $\cX_{2,\red}$, and similarly for  $\cX^{\underline{r}}_{\red}$ provided that $\underline{r}$ is regular;\ however, if $\underline{r}$ is irregular, then $\cX^{\underline{r}}_{\red}$ has codimension equal to the number of elements $i \in \Z/f\Z$ with $r_{i,1} = r_{i,2}$.
% We write
% \[\cX_{\red}^{\BT} = \bigcup_{\tau\, \textrm{tame}} \,\cX_{\red}^{\tau,\BT}\] where the union is taken over all tame inertial types $\tau$.

For a local Artinian $\cO$-algebra with finite residue field, let $\Rep_{\mathrm{fl},K'/K}(A)$ denote the category of  $A$-representations of $G_K$ that are potentially finite flat and become finite flat over $K'$. Let  $\Rep_{G_{K_\infty}}(A)$ denote the category of finite $A$-representations of $G_{K_{\infty}}$.

\begin{lemma}\label{lem:restriction-lemma}
Let $A$ be a local Artinian $\cO$-algebra with finite residue field, and $K'/K$ any tamely ramified finite extension.  The  functor $\Rep_{\mathrm{fl},K'/K}(A) \to \Rep_{G_{K_\infty}}(A)$ given by restriction from $G_K$-representations to $G_{K_\infty}$-representations is fully faithful.
\end{lemma}

\begin{proof}
For finite flat representations (as opposed to potentially finite flat), the lemma is now standard;\ see for example \cite[Thm.~3.4.3]{MR1971512}.

In the general case, let $V,W$ be objects of $\Rep_{\mathrm{fl},K'/K}(A)$, and let $f : V \to W$ be an $A$-linear map which is $G_{K_\infty}$-equivariant. We need to prove that $f$ is actually $G_K$-equivariant. Since $f$ is $G_{K'_\infty}$-equivariant, by the finite flat case it is $G_{K'}$-equivariant. Therefore $f$ is both $G_{K_\infty}$- and $G_{K'}$-equivariant. But $G_{K_\infty}$ and $G_{K'}$ generate all of $G_K$, because $K'/K$ is tamely ramified and any finite subextension of $K_\infty/K$ is totally wildly ramified. The lemma follows.
\end{proof}

\begin{lemma}\label{lemma:f-is-monomorphism}
For each tame type $\tau$ the map $\cX^{\tau,\BT} \rightarrow \cR_2$ given by the restriction of $f$  is a monomorphism.
\end{lemma}

\begin{proof}
    We follow the strategy of the proof of \cite[Prop.~7.2.11]{LLLM-models}. Namely, it suffices to show for any $a \ge 1$ and any finite type $\cO/\varpi^a$-algebra $A$ that the functor $\cX^{\tau,\BT}(A) \rightarrow \cR_2(A)$ is fully faithful. In the case that $A$ is a local Artinian $\cO$-algebra with finite residue field, this follows directly from Lemma~\ref{lem:restriction-lemma}, because $\cX^{\tau,\BT}(A)$ is equivalent to a full subcategory of the groupoid of  $A$-module representations of $G_K$ that are potentially finite flat and become finite flat over $K'$.  To see the latter, note that an object of $\cX^{\tau,\BT}(A)$ specializing to the Galois representation $\rhobar : G_K \to \GL_2(A/\mathfrak{m}_A)$ is pulled back from a versal morphism $\Spf(R_{\rhobar}^{\tau,\BT}) \to \cX^{\tau,\BT}$, where $R_{\rhobar}^{\tau,\BT}$ is a potentially Barsotti--Tate deformation ring. Therefore the corresponding Galois module becomes finite flat over $K'$, e.g.\ by~\cite[Prop.~2.3.8]{kis04}. %\mar{KK: Does this need a citation? RB: I've added a brief comment and I can flesh it out if necessary}

    The general case follows exactly as in the final paragraph of the proof of \cite[Prop.~7.2.11]{LLLM-models}: one establishes that for objects $x_1,x_2$ of $\cX^{\tau,\BT}(A)$ with images $y_1,y_2$ in $\cR_2(A)$, the functors $\Isom(x_1,x_2)$ and $\Isom(y_1,y_2)$ are representable by finite type $A$-schemes, and then one applies \cite[Lem~7.2.5]{LLLM-models} to reduce to the settled case of local Artinian $\cO$-algebras with finite residue field. 
\end{proof}

\begin{remark}
    We elaborate one point in the second part of the above argument. The reference to \cite[Prop.~5.4.8]{emertongeeproper} in  the proof of \cite[Prop.~7.2.11]{LLLM-models} handles the representability of $\Isom(y_1,y_2)$;\ it also establishes the representability of $\Isom(x^\circ_{1},x^\circ_2)$ where $x_i^\circ$ denotes the \'etale $\varphi$-module underlying $x_i$ (here we really do mean forgetting~$\Gamma$). It remains to check that commutation with $\Gamma$ cuts out a closed condition on $\Isom(x^\circ_{1},x^\circ_2)$. If it were the case that each projective \'etale $(\varphi,\Gamma)$-module were a direct summand of a free \'etale $(\varphi,\Gamma)$-module, as is the case for \'etale $\varphi$-modules by \cite[Lem.~5.2.14]{emertongeeproper}, 
    it would be straightforward to check this exactly as in the proof of \cite[Prop.~5.4.8]{emertongeeproper};\ but this does not seem immediately evident. On the other hand, writing each $x^\circ_i$ as the direct summand of a free \'etale $\varphi$-module, the action of $\Gamma$~(regarded as a \emph{semigroup}) on $x_i$ can be extended by zero to the free \'etale $\varphi$-module, and one can argue equally well using this extension.
\end{remark}

\begin{thm}\label{thm:EG-CEGS-isom}
There is an isomorphism $h^{\tau} : \cX^{\tau,\BT} \to \cZ^{\tau}$ which, for complete local Noetherian $\cO$-algebras $A$, is given by the identity on the corresponding Galois representations.
\end{thm}

\begin{proof}
Let $g : \cZ^\tau \to \cR_2$ be  the closed immersion $\cZ^{\tau} \to \cR^{\dd}$ followed by the isomorphism $\cR^{\dd} \to \cR_2$. Let $\cY$ sit at the corner of the pullback square
\[
  \begin{tikzcd}
 &\cY \arrow{r}{i_Z} \arrow[swap]{d}{i_X} & \cZ^{\tau} \arrow{d}{g} \\
 & \cX^{\tau,\BT} \arrow{r}{f} & \cR_{2}.
   \end{tikzcd}
\]
Since both $f,g$ are monomorphisms, so are $i_Z$ and $i_X$. In fact $g$ is a closed immersion, and therefore so is $i_X$;\ in particular $i_X$ is representable, hence representable by algebraic stacks in the sense of \cite[Def~3.1(2)]{EmertonFormal} (see also Remark~3.2 of \emph{loc. cit.}). It follows from \cite[Lem.~7.9]{EmertonFormal} that $\cY$ is a $p$-adic formal algebraic stack over $\Spf(\cO)$. Each of $\cX^{\tau,\BT}$, $\cZ^{\tau}$, and $\cY$ are topologically of finite type over $\cO$, e.g.\ because their special fibers are of finite type over $\F$.

Over any finite flat $\cO$-algebra $A$, the stacks $\cX^{\tau,\BT}$ and $\cZ^\tau$ have the same $A$-points, in the sense that $\cX^{\tau,\BT}(A)$ and $\cZ^\tau(A)$ each correspond to potentially Barsotti--Tate representations of type $\tau$ on projective $A$-modules. We deduce that functors $i_X(A)$ and $i_Z(A)$ are both essentially surjective.

Finally, $\cX^{\tau,\BT}$ and $\cZ^{\tau}$ are each flat over $\Spf(\cO)$, and they are each analytically unramified in the sense of \cite[Def.~8.22]{EmertonFormal}:\ as noted in the paragraph before \cite[Warning~7.2.1]{LLLM-models}, this is equivalent to having reduced versal rings at all finite type points, which follows from \cite[Cor~5.2.19]{cegsB} for  $\cZ^{\tau}$  and from \cite[Prop.~4.8.10]{EGmoduli} for $\cX^{\tau,\BT}$. (Recall that the deformation rings $R^{\tau,\BT}_{\rhobar}$ are reduced by definition.)

Taking all these observations together, we see that the two maps $i_X$ and $i_Z$ each satisfy the hypotheses of \cite[Lem.~7.2.6(1)]{LLLM-models}, hence  each is an isomorphism. We obtain an isomorphism by taking $h^{\tau} := i_Z \circ i_X^{-1} : \cX^{\tau,\BT} \to \cZ^{\tau}$. The statement about Galois representations then follows from the corresponding statements for $f$ and $g$ (together with full faithfulness of restriction from $G_K$ to $G_{K_\infty}$).
\end{proof}

As an application of Theorem~\ref{thm:EG-CEGS-isom}, we establish the existence and basic properties of loci in $\cZ^{\dd,1}$ of representations satisfying certain $p$-adic Hodge theoretic conditions.

\begin{cor}\label{cor:loci-in-Zdd}
Suppose $\underline{r}$ and $\tau$ are Hodge and inertial types with the property that no twist of a tr\`es ramifi\'ee representation has a potentially crystalline lift of type $\underline{r}$ and $\tau$. Then there is a unique reduced closed substack $\cZ^{\tau,\underline{r}}_{\red} \subset \cZ^{\dd,1}$ with the property that a representation $\rhobar : G_K \to \GL_2(\F')$  lies in $\cZ^{\tau,\underline{r}}_{\red}(\F')$ if and only if $\rhobar$ has a potentially crystalline lift of type $\underline{r}$ and $\tau$. {\upshape(}Here $\F'/\F$ is any finite extension.{\upshape)}

Furthermore $\cZ^{\tau,\underline{r}}_{\red}$ is equidimensional of dimension equal to that of $\cX^{\crys,\tau,\underline{r}}_{2,\red}$, and the irreducible components of $\cZ^{\tau,\underline{r}}_{\red}$ are in bijection with those of  $\cX^{\crys,\tau,\underline{r}}_{2,\red}$, such that corresponding components have the same finite type points.
\end{cor}

\begin{proof}
Uniqueness is immediate  from the fact that reduced closed substacks of reduced stacks are characterized by their finite type points. In particular the stack $\cZ^{\underline{r},\tau}_{\red}$ must be independent of the various choices in the construction that follows. 

The  reduced algebraic stack $\cX^{\crys,\tau,\underline{r}}_{2,\red}$ has  finitely many irreducible components $\cY_i$. By the hypothesis that no point of $\cX^{\crys,\tau,\underline{r}}_{2,\red}$ is a twist of a tr\`es ramifi\'ee representation, it follows that 
each $\cY_i$ is contained in the union $\cup_{\sigma} \cX^{\sigma}_{2,\red}$, the union taken over all non-Steinberg Serre weights;\ thus $\cY_i \subset \cX^{\sigma_i}_{2,\red}$ for some non-Steinberg Serre weight~$\sigma_i$. Let $\tau_i$ be any tame type such that $\sigma_i$ is a Jordan--H\"older factor of the reduction mod~$p$ of $\sigma(\tau_i)$, so that  $\cX^{\sigma_i}_{2,\red}$ (and therefore $\cY_i$) is contained in $\cX^{\tau_i,\BT}_{\red}$. Define $\cZ_i \subset \cZ^{\tau,1}$ to be the image of $\cY_i$ under the isomorphism $h^{\tau_i}$ of Theorem~\ref{thm:EG-CEGS-isom}. Then the (reduced) union of the $\cZ_i$'s inside $\cZ^{\dd,1}$ has the desired property.

By construction $\cY_i$ and $\cZ_i$ are isomorphic and have the same finite type points. Since each $\cY_i$ is irreducible and there are no inclusions between the $\cY_i$'s, consideration of finite type points shows that the same must be true of the $\cZ_i$'s inside $\cZ^{\dd,1}$. This gives the final statement.
\end{proof}

\subsection{Irregular loci in \texorpdfstring{$\cZ^{\dd,1}$}{Z\^{}dd,1}}\label{sec:combinatorial-results}
We now resume the running assumption that the extension $K/\Qp$ is unramified, and introduce the following terminology.

\begin{defn}\label{def:hodge-type-names}
We say that the Hodge type~$\underline{r}$ is 
\begin{itemize}
\item \textit{$p$-bounded} if $r_{i,1} - r_{i,2} \leq p$ for all $i$, %\mar{DS: I'm not a fan of ``small''. How do we feel about $p$-restricted even though irregular is permitted? RB: ``closely spaced'', maybe? DS: how about $p$-bounded?}
\item \textit{Steinberg} if  $r_{i,1} - r_{i,2} = p$ for all $i$, and 
\item \textit{regular} if $r_{i,1} - r_{i,2} > 0$ for all $i$. \end{itemize} 
\end{defn}
\begin{lem}\label{lem:p-bounded-ns} Assume that the Hodge type $\underline{r}$ is $p$-bounded and non-Steinberg. Then~$\underline{r}$ together with the trivial type satisfies the hypothesis of Corollary~\ref{cor:loci-in-Zdd}, giving a reduced closed substack $\cZ^{\underline{r}}_{\red} := \cZ^{\mathrm{triv},\underline{r}}$ of $\cZ^{\dd,1}$ whose finite type points are precisely those having a crystalline lift of type $\underline{r}$.
\end{lem}

\begin{proof}
Suppose that $\rhobar : G_K \to \GL_2(\F')$ has a crystalline lift of Hodge type $\underline{r}$. We must show that $\rhobar$ is not tr\`es ramifi\'ee.
If $\underline{r}$ is regular, then there is a non-Steinberg Serre weight $\sigma$ such that $\rhobar$ has a crystalline lift of Hodge type $\underline{r}$ if and only if $\sigma \in W(\rhobar)$, and we conclude by \cite[Lem.~A.5(2)]{cegsC}. So we may assume for the remainder of the proof that $\underline{r}$ is irregular.

If $\rhobar$ is reducible, then by \cite[Chapter 5, Cor~2.7]{HWThesis}  the ratio of the diagonal characters of $\rhobar$ has restriction to inertia equal to  $\prod_{i=0}^{f-1} \omega_i^{t_i}$, where 
for some subset $J \subset \Z/f\Z$ we have $t_i = r_{i,1}-r_{i,2}$ for $i \in J$ and $t_i = r_{i,2} - r_{i,1}$ for $i \not\in J$. In particular $t_i \in [-p,p]$ for all $i$. We will show that  the ratio of characters cannot be cyclotomic, and indeed that we cannot have 
\begin{equation}\label{eq:t-not-cyclotomic} \sum_{i=0}^{p-1} p^{f-1-i} t_i \equiv \tfrac{p^f-1}{p-1} \pmod{p^f-1}. 
\end{equation}
Since $\underline{r}$ is irregular, we have $t_j=0$ for some $j$, and without loss of generality (e.g.\ multiplying both sides by $p^{f-1-j}$) we may assume $t_{f-1}=0$. Then the left-hand side of \eqref{eq:t-not-cyclotomic}, considered as an integer, is divisible by $p$;\ while the right-hand side, again considered as an integer, is $1$ modulo $p$. We must therefore have an equality
\[ \sum_{i=0}^{p-1} p^{f-1-i} t_i = \tfrac{p^f-1}{p-1} + k(p^f-1) \]
where $k \equiv 1 \pmod{p}$. We have $\left| \sum_{i=0}^{p-1} p^{f-1-i} t_i \right| < p \cdot \tfrac{p^f-1}{p-1}$, with strict inequality because $t_{f-1} = 0$. This already rules out $k \ge 1$. Similarly
\[ (1-p)(p^f-1) + \tfrac{p^f-1}{p-1} = (2p-p^2) \tfrac{p^f-1}{p-1} \le -p \cdot  \tfrac{p^f-1}{p-1}\]
since $p \ge 3$,
so $k \le 1-p$ is ruled out as well. Therefore \eqref{eq:t-not-cyclotomic} has no solutions.
% This can be checked explicitly using \cite[Chapter 5, Theorem 2.5]{HWThesis} and \cite[Lemma 7.1]{gls12}. The first result can be used to describe the possible restrictions to inertia of reductions of crystalline representations with specified $p$-bounded, possibly non-distinct, Hodge--Tate weights. The second is a lemma that allows one to check congruences modulo $p^f-1$, from which one then can conclude that the required congruences can not be satisfied by extensions of the cyclotomic by the trivial character.
\end{proof}

To each non-scalar tame type $\tau$ and profile $J \subset \Z/f'\Z$, we will now associate a Hodge type $r(\tau,J)$, or more precisely a Hodge type up to equivalence under an equivalence relation that we will define in the next two paragraphs.

Let us write $\Lambda \subset \Z^{f}$ for the set of tuples $\underline{\lambda} = (\lambda_i)$ such that the inertial character $\prod_{i=0}^{f-1} \omega_i^{\lambda_i}$ is trivial. Concretely, this is the set of tuples $\underline{\lambda}$ such that 
\[ \sum_{i=0}^{f-1} p^{f-i} \lambda_i \equiv 0 \pmod{p^f-1}. \]
Interpreting $\underline{\lambda}$ as a Hodge type of rank $1$, we see that  $\Lambda$ can equivalently be described as the set of Hodge types of crystalline characters of $G_K$ that are trivial  modulo $p$.

\begin{defn}\label{def:equiv-on-Hodge-types} If $\underline{r}$ is a Hodge type and $\underline{\lambda} \in \Z^f$, we define  $\underline{r} + \underline{\lambda}$ to be the Hodge type $\{ r_{i,j} + \lambda_i \}_{i,j}.$ We define an equivalence relation $\sim$ on the set of Hodge types by taking $\underline{r} \sim \underline{r}'$ if and only if $\underline{r}' = \underline{r} + \underline{\lambda}$ with $\lambda \in \Lambda$. If $\underline{r} \sim \underline{r}'$ then evidently $\cX^{\underline{r}}_{\red} = \cX^{\underline{r}'}_{\red}$, thanks to the description of $\Lambda$ in terms of crystalline characters with trivial reduction mod $p$.
\end{defn}

To the pair $\tau$ and $J$, we have already associated a tuple of integers $s_{J,i}$ and a character $\Theta_J$, as in equations \eqref{def:s_J} and \eqref{def:theta_J}. Write
\[ \Theta_J = \prod_{i=0}^{f-1} \omega_i^{\theta_{J,i}} \] for some integers $\theta_{J,i}$. The tuple $\theta_J = (\theta_{J,i})_i$ is not uniquely defined, but it is unique up to translation by an element of $\Lambda$. We then define $r(\tau,J)$ 
by the formula
\begin{align}
r(\tau, J) := \{-s_{J, i} - \theta_{J, i}, 1 - \theta_{J, i}\}_{i \in \Z/f\Z}
\end{align}
and obtain a Hodge type up to equivalence under $\sim$. Recall from Sections~\ref{sec:notation:HT} and~\ref{sec:ill-sw} that in our conventions, and for profiles $J \in \cP_{\tau}$, we have $\sigmabar(\tau)_{J} \in W(\rhobar)$ if and only if $\rhobar$ has a crystalline lift of Hodge type $r(\tau,J)$.

\begin{defn}
We define $\cN^\tau(J)$ to be the stack $\cZ^{\underline{r}}_{\red}$ for any representative $\underline{r}$ of $r(\tau,J)$. This is well-defined by Lemma~\ref{lem:p-bounded-ns} and the final sentence of Definition~\ref{def:equiv-on-Hodge-types}.
\end{defn}

If $J \in \cP_{\tau}$ then $\cN^{\tau}(J) = \overline{\cZ}(\sigmabar(\tau)_J)$ is an irreducible component of $\cZ^{\dd,1}$. However, if instead $J \not\in \cP_{\tau}$ then the finite type points of $\cN^{\tau}(J)$ are the representations with crystalline lifts of (irregular) Hodge type $r(\tau,J)$, and $\cN^{\tau}(J)$ has codimension $\#\{ i : s_{J,i} = -1 \}$. The following combinatorial result shows that $\cZ^{\underline{r}}_{\red}$ arises as one of the loci $\cN^{\tau}(J)$ for every $p$-bounded and non-Steinberg Hodge type $\underline{r}$.

\begin{prop}\label{prop:existence-tame-type}
For each $p$-bounded and non-Steinberg Hodge type $\underline{r}$, we can find a non-scalar tame type $\tau$ and a profile $J \subset \Z/f'\Z$ such that $\underline{r} \sim r(\tau, J)$, or equivalently such that $\cZ^{\underline{r}}_{\red} = \cN^\tau(J)$.
\end{prop}

Recall from \cite{cegsC} that for $i \in \bZ/f'\bZ$ and each profile $J$, we say that $(i-1, i)$ is a \textit{transition} (at $i$) if $\# \{i-1, i\} \cap J = 1$, and a \textit{non-transition} otherwise.

\begin{proof} Set $s_i = r_{i,1} - r_{i,2} - 1 \in [-1,p-1]$. Recall that for each non-scalar tame type $\tau = \eta \oplus \eta'$ and each profile $J$ there is an associated tuple of integers $s_{J,i}$ given by the formula
\begin{align}\label{def:s_J_redux}
&s_{J, i} := \begin{cases}
p-1-\gamma_i - \delta_{J^{c}}(i) &\text{ if } i-1\in J, \\
\gamma_i - \delta_J(i) &\text{ if } i-1 \not\in J,
\end{cases}
\end{align}
with the integers $\gamma_i \in [0,p-1]$ coming from writing the ratio $\eta/\eta'$ in terms of multiplicative lifts of fundamental characters as in \eqref{eq:gamma_i}. It suffices to produce a non-scalar tame type $\tau$ and a profile $J$ such that $s_i = s_{J,i}$ for all $i$, for then we will have  $\underline{r} = r(\tau \otimes \chi, J)$ for a suitable character $\chi$ of $G_K$.

There are evidently $p^f$ possibilities for the tuple $\gamma_0,\ldots,\gamma_{f-1}$. (Recall that in the cuspidal case we then have $\gamma_{i+f} = p-1-\gamma_i$.) Each of these $p^f$ possibilities can arise from some cuspidal type $\eta \oplus \eta^{p^f}$. To see this, note for example that there are $p^{2f} - p^f$ possibilities for the character $\eta$, and each possibility for the ratio $\eta/\eta^{p^f}$ arises from $p^f-1$ different $\eta$'s. In the principal series case every possibility can arise except the case $\gamma_i = 0$ for all $i$, or the case $\gamma_i = p-1 $ for all $i$, since these would give a scalar type.

We now define a tuple $\gamma_0,\ldots,\gamma_{f-1}$ and a profile $J$ by the following procedure. First, choose arbitrarily whether or not $-1$ lies in $J$. Then for each $0 \le i \le f-1$ in turn, we proceed as follows.
\begin{itemize}
\item If $s_i \in [0,p-2]$, we choose arbitrarily whether or not $i \in J$, and then define $\gamma_i$ by the formula     
\begin{align}\label{def:gamma_i-good}
&\gamma_i := \begin{cases}
p-1-s_i - \delta_{J^{c}}(i) &\text{ if } i-1\in J, \\
s_i + \delta_J(i) &\text{ if } i-1 \not\in J,
\end{cases}
\end{align}
\item If instead $s_i \in \{-1,p-1\}$, then exactly one of the two possibilities for $i \in J$ or $i \not\in J$ in \eqref{def:gamma_i-good} yields a $\gamma_i$ in the range $[0,p-1]$, and we make that choice. Observe for what follows that if $s_i = -1$ this requires making $(i-1,i)$ a transition, and if $s_i = p-1$ this requires making $(i-1,i)$ a non-transition.
\end{itemize}
Finally, if at the end of this procedure we have both $-1,f-1\in J$ or both $-1,f-1 \not\in J$ then we can form a principal series type with profile $J$ and yielding the chosen integers $\gamma_0,\ldots,\gamma_{f-1}$;\ while if $-1 \in J$ and $f-1 \not\in J$ or vice-versa then $J$ can be extended to a profile for a cuspidal type yielding the chosen integers $\gamma_0,\ldots,\gamma_{f-1}$.

This completes the construction, except for the possibility that by following the above procedure we may have constructed a scalar principal series type. We check that this can always be avoided. If at least one $s_i$ lies in the range $[0,p-2]$ then there is at least one $0 \le i < f$ where we could freely choose to create either a transition or a non-transition at $i$;\ therefore we may arrange to create an odd total number of transitions among $i$ with $0 \le i < f$, and thereby construct a cuspidal type. 

Thus we are reduced to the case where $s_i \in \{-1,p-1\}$ for all $i$, with an even number of transitions and therefore an even number of $-1$'s.
Since $\underline{r}$ is non-Steinberg, we do not have $s_i = p-1$ for all $i$, and therefore there at least two transitions. Let $i < i'$ be two consecutive transitions;\ that is, $j$ is a non-transition for each $i < j < i'$. Then $i \in J$ and $i' \not \in J$, in which case $\gamma_i = 0$ and $\gamma_{i'} = p-1$;\ or else the vice-versa. In either case there is at least one $0$  and at least one $p-1$ among the $\gamma_i$'s, and the type we have constructed is non-scalar.
\end{proof}

\begin{remark}\label{rem:cant-choose-j-when} Suppose $j$ satisfies $r_{j,1} - r_{j,2} \in [1,p-1]$. Then in Proposition~\ref{prop:existence-tame-type} we are always able to choose the tame type $\tau$ and profile $J$ so that for  this fixed $j$, there is either a transition at $j$ or not, as desired, with the following exceptions:
\begin{enumerate}
\item If $f=1$ and $r_{j,1} - r_{j,2} = 1$, then $(j-1, j)$ is forced to be a transition.
\item If $f\geq 2$ and 
\[
r_{i,1} - r_{i,2} =
\begin{cases}
p-1 &\text{ if } i = j ,\\
0 &\text{ if }  i =j+1, \text{ and }\\
p &\text{ otherwise }
\end{cases}
\]
then $(j-1, j)$ is forced to be a non-transition.
\end{enumerate}
Indeed, according to the proof of Proposition~\ref{prop:existence-tame-type}, the only possible obstruction is if there is a single $j$ satisfying $r_{j,1} - r_{j,2} \in [1,p-1]$, and for one choice for whether or not there is a transition at $j$, our construction leads to a scalar principal series type. Then we are obligated to make the opposite choice. By an analysis similar to the argument in the last paragraph of the  proof of Proposition~\ref{prop:existence-tame-type}, the construction leading to the scalar principal series type must have either 0 or 2 transitions, and in the latter case~$j$ must be one of the two transitions. 

It therefore suffices to evaluate \eqref{def:s_J_redux} with $\gamma_i = 0$ for all $i$ or $\gamma_i = p-1$ for all~$i$, and for $J \subset \Z/f\Z$ of the form $\varnothing$ or $\{i',\ldots,j-1\}$ or their complements, and confirm which ones lead to a single $s_{J,i} = r_{i,1} - r_{i,2} - 1$ being in the range $[0,p-2]$. This leads to the two exceptions listed above. For example, the  exception (2) comes from choosing $\gamma_i = 0$ for all $i$ and $J = \{j\}^c$, as well as from choosing $\gamma_i = p-1$ for all $i$ and $J = \{j\}$.
\end{remark}

\subsection{Comparison of irregular loci and scheme-theoretic images of vertical components} \label{sec:comparison-section}
Our goal in this subsection is to establish the following, which shows that each $\cZ^{\tau}(J) \subset \cZ^{\dd,1}$ can be described as the closed substack whose finite type points admit certain crystalline lifts. 

\begin{theorem}\label{thm:irregular-locus} If $\tau$ is a non-scalar tame type and $J \subset \Z/f'\Z$ is any profile, then $\cN^{\tau}(J) = \cZ^{\tau}(J)$;\ in other words, the finite type points of $\cZ^{\tau}{(J)}$ are precisely the representations with crystalline lifts of Hodge type $r(\tau,J)$.
\end{theorem}

\begin{cor}
    \label{cor:Z-tau-J is combinatorial} The stack $\cZ^{\tau}(J)$ depends only on the Hodge type $r(\tau,J)$.
\end{cor}

Theorem~\ref{thm:irregular-locus} was proved in \cite{cegsC} in the regular case (i.e., when $J \in \cP_\tau$), but is new in the irregular case.
We also note the following application of Theorem~\ref{thm:irregular-locus} to the Emerton--Gee stacks;\  again this result is new in the irregular case.

\begin{cor}\label{cor:irreducibility-of-EG-irregular}
Suppose the Hodge type $\underline{r}$ is $p$-bounded and non-Steinberg. Then $\cX^{\underline{r}}_{\text{red}}$ is irreducible. 
\end{cor}

\begin{proof}
This follows by combining Proposition \ref{prop:existence-tame-type}, Theorem \ref{thm:irregular-locus},  the irreducibility of the stacks $\cZ^{\tau}(J)$, and the last part of Corollary~\ref{cor:loci-in-Zdd}.
\end{proof}

Note that Corollary~\ref{cor:irreducibility-of-EG-irregular} is false for  Hodge types that are Steinberg, because if $\underline{r}$ is Steinberg then $\cX^{\underline{r}}_{\text{red}}$ is the union of two irreducible components (\emph{cf}.~Theorem~\ref{thm:cegs-main-X}(2)).

The proof of Theorem~\ref{thm:irregular-locus} will occupy the remainder of this subsection. The strategy is to prove that $\cN^\tau(J)$ and $\cZ^\tau(J)$ are equidimensional of the same dimension, and that there is an inclusion $\cN^\tau(J) \subset \cZ^\tau(J)$. Since $\cZ^\tau(J)$ is irreducible by construction, the inclusion must be an equality. 

Recall that $S^\tau(J)$ is defined to be the set $\{i \in \Z/f\Z : s_{J,i} = -1\}$;\ or equivalently if $\underline{r} \sim r(\tau,J)$, then $S^{\tau}(J)$ is the set $\{ i \in \Z/f\Z : r_{i,1} = r_{i,2} \}$.  By the last part of Corollary~\ref{cor:loci-in-Zdd}, we already know that $\cN^{\tau}(J)$ is equidimensional of dimension $[K:\Q] - |S^\tau(J)|$. We begin by checking that the same is true of $\cZ^{\tau}(J)$.

\begin{prop}\label{prop:dim-Z-tau-J}
We have $\dim \cZ^{\tau}(J) = [K:\Qp] - |S^{\tau}(J)|$. 
\end{prop}
\begin{proof} 
The proof is an elaboration of the proof of \cite[Thm.~5.1.12]{cegsC}. The construction of the stack $\cZ^\tau(J)$ in \cite[Def.~4.2.12]{cegsC}  furnishes a map 
\[\xi : \Spec \Bkfree \to \cZ^{\tau}(J)\]
where $\xi$ is scheme-theoretically dominant and the  source has dimension $[K:\mathbb{Q}_p]+2$. 

We freely use the notation of \cite[\S3.3]{cegsC}. Assume first that $|S^{\tau}(J)| < f$, and let $X$ be the dense open subscheme of $\Spec B^{\kfree}$ defined in the paragraph preceding \cite[Rmk.~3.3.8]{cegsC} (and also denoted $X$ there). By \cite[\href{https://stacks.math.columbia.edu/tag/0DS4}{Tag 0DS4}]{stacks-project}, it suffices to show that 
the restriction of $\xi$ to $X$ has fibers of dimension $|S^{\tau}(J)|+2$ in the sense of \cite[\href{https://stacks.math.columbia.edu/tag/0DRL}{Tag 0DRL}]{stacks-project}.

Given an $A^{\kfree}$-algebra $A$, an $A$-point of $X$ is an extension class % tuple $(p, [\gE])$, where $p: \Akfree \to A$ is an $\F$-algebra homomorphism and %$[\gE]$ is an extension class in 
\[[\gE] \in \Ext^{1}_{\cK(A)}(\gM(J)_{A, \overline{x}}, \gN(J)_{A, \overline{y}})\] which does not become the trivial class on inverting $u$ after any base change,
where $\gM(J), \gN(J)$ are as in \cite[Def.~4.2.8]{cegsC}, $\gM(J)_{A,\overline{x}}, \gN(J)_{A,\overline{y}}$ are ``unramified twists'' as in \cite[Def.~3.3.2]{cegsC}, and $\overline{x},\overline{y}$ denote the images of $x,y \in A^{\kfree}$ in $A$.

%Equivalently $[\gE_{\text{univ}}]$ is the data of the extension witnessed by a short exact sequence
%\[0 \to \gN(J)_{\Akfree, y} \xrightarrow{\iota} \gE_{\text{univ}} \xrightarrow{\pi} \gM(J)_{\Akfree, x} \to 0.\]

Setting $U_{\Akfree} := \kExt^{1}_{\cK(\Akfree)}(\gM(J)_{\Akfree, x}, \gN(J)_{\Akfree, y})$ as in the discussion before \cite[Rmk.~3.3.8]{cegsC}, $Y:=\Spec \Akfree[U_{\Akfree}^{\vee}]$ is a closed subscheme of $\Spec \Bkfree$ whose $A$-points are extensions
%tuples $(p, [\gK])$ where $p: \Akfree \to A$ is an $\F$-%algebra homomorphism and
$[\gK] \in \Ext^{1}_{\cK(A)}(\gM(J)_{A, \overline{x}}, \gN(J)_{A, \overline{y}})$  that do become trivial upon inverting $u$. There is a map
\begin{equation}\label{eq:ker-ext-not-all}
(X \times_{\Spec \Akfree} Y) \times_{\F} \Gm \times_{\F} \Gm \to X \times_{\Spec \Akfree \times_{\F} \cZ^{\tau}(J)} X \to  X \times_{\cZ^{\tau}(J)} X 
\end{equation}
given by mapping an $A$-point $([\gE], [\gK], r, s)$ in the domain to 
\begin{itemize}
\item the extension $[\gE]$ in the first coordinate,
\item the extension $[\gE'] := r \cdot [\gE] + [\gK]$ in the second coordinate, along with 
%where $[\gE']$ is the extension class $r[\gE] + [\gK]$. Note that if $[\gE]$ corresponds to a short exact sequence
% \[0 \to \gN(J)_{A, p(y)} \xrightarrow{\iota} \gE \xrightarrow{\pi} \gM(J)_{A, p(x)} \to 0, \] then $r[\gE]$ is the extension class witnessed by the short exact sequence
% \[0 \to \gN(J)_{A, p(y)} \xrightarrow{r^{-1} \iota} \gE' \xrightarrow{\pi} \gM(J)_{A, p(x)} \to 0. \]
\item the data of an isomorphism $\gE[1/u] \cong \gE'[1/u]$
which on the quotients $\gM(J)_{A,\overline{x}}[1/u]$ is induced from multiplication by $s$ on  $\gM(J)$.
\end{itemize}
 Note that the automorphisms of $\gE[1/u]$ are a torsor for $\Gm$:\ this follows from \cite[Lem.~3.3]{cegsC} and the fact that $\gE[1/u]$ is non-split after any base change.
 
It is immediately verified that the map \eqref{eq:ker-ext-not-all} is a monomorphism and a bijection on finite type points, observing that
\[ X \times_{\Spec \Akfree \times_{\F} \cZ^{\tau}(J)} X \to  X \times_{\cZ^{\tau}(J)} X \]
is bijective on finite type points by \cite[Lem.~3.3.5]{cegsC} and the fact that points of $X$ remain nonsplit after inverting $u$.

Therefore, if $\F'$ is a finite extension of $\F$, the fiber of $\xi$ over an $\F'$-point of $X$ admits a surjective monomorphism from
\[\Spec\F'[\kExt^{1}_{\cK(\F')}(\gM(J)_{\F', \overline{x}}, \gN(J)_{\F', \overline{y}})^{\vee}] \times_{\F} \Gm \times_{\F} \Gm\] which has dimension $|S^{\tau}(J)| + 2$ by a comparison between \cite[Prop.~5.1.8]{cegsC} and \eqref{def:s_J}. (The restriction to $B^\kfree$ avoids the exceptional case of \cite[Prop.~5.1.8]{cegsC}.) Thus, the dimension of the fiber is also $|S^{\tau}(J)| + 2$ by an application of \cite[\href{https://stacks.math.columbia.edu/tag/0DS4}{Tag 0DS4}]{stacks-project}.
 
Now, suppose $|S^{\tau}(J)| = f$, or equivalently, 
\[\kExt^{1}_{\cK(A)}(\gM(J)_{A,\overline{x}}, \gN(J)_{A,\overline{y}}) = \Ext^{1}_{\cK(A)}(\gM(J)_{A,\overline{x}}, \gN(J)_{A,\overline{y}})\] for all 
$\Akfree$-algebras $A$. In this case, the map~$\xi$ factors as
\[\xi: \Spec \Bkfree \to \Spec \Akfree \xrightarrow{g} \cZ^{\tau}(J)\] where the first arrow is the structure map and $g$ maps the universal point of $\Spec \Akfree$  to $\gM(J)_{\Akfree, x}[1/u] \oplus \gN(J)_{\Akfree, y}[1/u]$. 

Let $p_0 \in \Spec \Akfree(\F')$ be fixed. The fiber of $g$ over $g(p_0)$ is given by $\Spec \F' \times_{\cZ^{\tau}(J)} \Spec \Akfree$. For a finite type $\F'$-algebra $A$, an $A$-point of this fiber is a $A^\kfree$-algebra structure $p : A^\kfree \to A$ together with an isomorphism 
\[\gM(J)_{A,p(x)}[1/u] \oplus \gN(J)_{A,p(y)}[1/u] \xrightarrow{\sim} \gM(J)_{A, p_0(x)}[1/u] \oplus \gN(J)_{A, p_0(y)}[1/u].\] 
It follows that the fiber of $g$ over $g(p_0)$ admits a monomorphism, bijective on finite type points, 
either from $\Gm \times_\F \Gm$ or else from the disjoint union of two copies of $\Gm \times_\F \Gm$. One copy comes from $A$-points with $A^\kfree$ algebra structure $p : A^{\kfree} \longnto{p_0} \F' \to A$, so that $p(x) = p_0(x)$ and $p(y) = p_0(y)$;\ the other copy, if it exists, comes from the unique $A^{\kfree}$-algebra structure $p$ such that
 \begin{align*}
 \gM(J)_{A, p(x)}[1/u] & \cong \gN(J)_{A, p_0(y)}[1/u]\\
\gN(J)_{A, p(y)}[1/u] & \cong \gM(J)_{A, p_0(x)}[1/u],
 \end{align*} 
again if it exists. Note that because we are on $A^\kfree$, there is no isomorphism $ \gM(J)_{A, p_0(x)}[1/u] \cong \gN(J)_{A, p_0(y)}[1/u]$, so there is no $\GL_2$ in the fiber.

%  These isomorphisms uniquely determine $p$ in terms of $p_0$, and in fact if such a $p$ exists, it must descend to an $\F'$-point of $\Spec \Akfree$.

% There exists a map
% \begin{align*}
% \Gm \times_{\F} \Gm \to \Spec \F' \times_{\cZ^{\tau}} \Spec \Akfree
% \end{align*}
% given by sending $(r, s)$ in the domain to $(p_0, \alpha)$ where $\alpha$ is the automorphism given by scaling $\gM(J)_{A, p_0(x)}[1/u]$ by $r$ and $\gN(J)_{A, p_0(y)}[1/u]$ by $s$. This map is evidently a monomorphism. Suppose there does not exist any finite type point $p \neq p_0$ of $\Spec \Akfree$ such that (\ref{eqn:fiber1}) and (\ref{eqn:fiber2}) hold, then this map is a bijection on finite type points.

% On the other hand, if there exists $p \neq p_0$ such that (\ref{eqn:fiber1}) and (\ref{eqn:fiber2}) hold, then the fiber admits another monomorphism from $\Gm \times_{\F} \Gm$ given by sending $(r, s)$ in the domain to $(p, \alpha')$ where $\alpha'$ is the scaling of some (any) fixed isomorphisms in (\ref{eqn:fiber1}) and (\ref{eqn:fiber2}) by $r$ and $s$ respectively. We thus get a map from a disjoint union of two copies of $\Gm \times_{\F} \Gm$, which is also a monomorphism and induces a bijection on finite type points.

%In either scenario, by \cite[\href{https://stacks.math.columbia.edu/tag/0DS4}{Tag 0DS4}]{stacks-project}, the dimension of the fiber is $2$ and the dimension of $\cZ^{\tau}(J)$ is~$0$.

By density of finite type points, the topological spaces associated to the scheme-theoretic images of the one or two copies of $\Gm \times_{\F} \Gm$ are precisely the irreducible components of $|\cZ^{\tau}(J)|$, and by \cite[\href{https://stacks.math.columbia.edu/tag/0DS4}{Tag 0DS4}]{stacks-project}, each scheme-theoretic image has dimension $2$. Thus, using \cite[\href{https://stacks.math.columbia.edu/tag/0DRZ}{Tag 0DRZ}]{stacks-project}, the dimension of the fiber is $2$. Another application of \cite[\href{https://stacks.math.columbia.edu/tag/0DS4}{Tag 0DS4}]{stacks-project} then shows that the dimension of $\cZ^{\tau}(J)$ is $0$.
\end{proof}

We recall the following, which (more or less) is  one of the main results of  \cite{gls12}.

\begin{thm}\label{thm:gls-irregular}
Let $\F'/\F$ be a finite extension. Suppose that   $\rhobar : G_K \to \GL_2(\F')$ is a Galois representation, and that $M^0 \in \cR_2(\F')$ is an \'etale $\varphi$-module  for $K$ (without descent data) such that $\rhobar |_{G_{K_{\infty}}} \cong T_K(M^0)$.
If $\rhobar$ has a crystalline lift with $p$-bounded Hodge type $\underline{r}$, then 
$(M^0)_i$ has a basis $x_i,y_i$ for each $i \in \Z/f\Z$ such that the partial Frobenius maps $\Phi_{M^0,i}$, written with respect to the basis $(x_i,y_i)_{i \in \Z/f\Z}$,
have matrices
\[
    B_{i} \begin{pmatrix}
    v^{r_{i,1}}  & 0 \\
    0 & v^{r_{i,2}}
    \end{pmatrix}
\]
for some $B_i \in \GL_2(\F'[\![v]\!])$.
\end{thm}

\begin{remark}\label{rem:again-order-irrelevant}
As in Remark~\ref{rem:ordering-irrelevant}, the theorem remains true for any reordering of the diagonal elements in the matrices $\left(\begin{smallmatrix}
    v^{r_{i,1}} & 0 \\ 0 & v^{r_{i,2}}
\end{smallmatrix}\right)$.
\end{remark}
\begin{proof}
Twisting by a character we reduce to the case where each $r_{i,j} \in [-p,0]$. This case will follow from \cite[Thm.~4.22]{gls12} after translating between the conventions of \cite{gls12} and \cite{cegsC} for Galois representations and Hodge--Tate weights. More precisely:\  from Lemma~\ref{lem:gal-rep-comparison} recall that $T_K(M^0) \cong T_K^*((M^0)^\vee)$. On the other hand, having crystalline lifts with Hodge type $\underline{r}$ in the conventions of \cite{cegsC} (i.e., in our conventions) is equivalent to having crystalline lifts with Hodge type  $-\underline{r} := \{-r_{1,i}, -r_{2,i}\}_{i \in \Z/f\Z}$ in the conventions of \cite{gls12}. Thus $T_K(M^0)$ has a crystalline lift with Hodge type $\underline{r}$ in our sense if and only if, in the sense of \cite{gls12}, $T_K^*((M^0)^\vee)$ has a crystalline lift with Hodge type $-\underline{r}$. Then \cite[Thm.~4.22]{gls12} tells us that $(M^0)^\vee$ admits a basis for which the partial Frobenius maps have matrices
\begin{align*}
B_{i} \begin{pmatrix}
v^{-r_{1, i}} & 0 \\
0 & v^{-r_{2,i}}
\end{pmatrix}
\end{align*}
for some $B_i \in \GL_2(R[\![v]\!])$. Finally, by Remark~\ref{rem:frob-dual} the partial Frobenius matrices for $M^0$ are the inverse transpose of those for $M^0$.
\end{proof}

Theorem \ref{thm:irregular-locus} now follows easily from all the work we have already done.

% \begin{theorem}\label{thm:irregular-locus} If $\tau$ is a non-scalar tame type, then $\cN^{\tau}(J) = \cZ^{\tau}(J)$.
% \end{theorem}
\begin{proof}[Proof of Theorem \ref{thm:irregular-locus}]
% Explicitly, we need to show that the finite type points of $\cZ^{\tau}(J)$ are precisely the mod $p$ representations with crystalline lifts with Hodge--Tate weights in the $\bar{i}$-th embedding given by
% $-s_{J, i} - \lambda_{J, i}$ and $1 -\lambda_{J, i}$
% for any choice of $i$ lifting $\bar{i}$.

% \begin{comment}
% \textcolor{red}{BL: Is there a twist by a character which is trivial mod p happening here? It should be possible to also rewrite 3.3 in such a way that we get matching on the nose} \textcolor{brown}{KK: edited Prop 3.3 accordingly.}
% \end{comment}

Theorem~\ref{prop:modified-basis-of-invariants} and Theorem~\ref{thm:gls-irregular} show that $\cN^{\tau}(J) \subset \cZ^{\tau}(J)$. By \cite[Thm.~4.8.14]{EGmoduli} and Proposition \ref{prop:dim-Z-tau-J}, $\cN^{\tau}(J)$ and $\cZ^{\tau}(J)$ have the same dimension. Since $\cZ^{\tau}(J)$ is irreducible, being the scheme-theoretic image of the irreducible component $\cC^{\tau}(J)$ of $\cC^{\tau,\BT,1}$, the theorem follows.
\end{proof}

%Next, we consider which $r = \{r_{1, \bar{i}}, r_{2, \bar{i}}\}_{\bar{i} \in \bZ/f\bZ}$ equal $r(\tau, J)$ for some $\tau$ and $J$. 

%Moreover, the pair $(\tau,J)$ can be chosen so that if $j \in \bZ/f'\bZ$ is fixed, then $j \in J$ (resp. $j \not\in J)$.

%, and further, for each $\bar{i}$ with $\Delta(r)_{\bar{i}} \in [1, p-1]$, we can freely choose $(i-1, i)$ to be a transition or not when $i$ is in a fixed set of representatives of $\bar{i}$.}
\begin{comment}
\begin{enumerate}
\item For each $\bar{i}$, $r_{1, \bar{i}} - r_{2, \bar{i}} \in \{0, p\}$ and $\#S(\rhobar)$ is odd.
\item For each $\bar{i}$, $r_{1, \bar{i}} - r_{2, \bar{i}} = p$.
\item $f=1$, and $r_{1, 0} - r_{2, 0} = 1$.
\item There exists $\bar{j} \in \Z/f\Z$ such that $r_{1, \bar{j}} - r_{2, \bar{j}} = 0$, $r_{1, \bar{j}-1} - r_{2, \bar{j}-1} = p-1$, and for each $\bar{i} \not\in \{\bar{j}-1, \bar{j}\}$, $r_{1, \bar{i}} - r_{2, \bar{i}} =p$. (Thus, $S(\rhobar) = \{\bar{j}\}$.) 
\end{enumerate}
Moreover, whenever neither of the above holds, a non-trivial tame principal series $\tau$ can be chosen so that $j \in J$ (resp. $j \not\in J)$.
\end{comment}

\begin{remark}\label{rem:converse-remark-regular}
    If the Hodge type $\underline{r}$ is $p$-bounded and \emph{regular}, then under the additional hypothesis that $\rhobar$ is not a twist of a tr\'es ramifi\'ee representation, the results in \cite[\S\S7--9]{gls12} can be reinterpreted as providing a converse to Theorem~\ref{thm:gls-irregular};\ that is, if $\rhobar$ is not a twist of a tr\'es ramifi\'ee representation, then the if-then of Theorem~\ref{thm:gls-irregular} is in fact an if-and-only-if. (The extra hypothesis on $\rhobar$ is necessary because if $\rhobar$ is tr\'es ramifi\'ee then $\rhobar|_{G_{K{\infty}}}$ can be split, and in that case $\rhobar|_{G_{K{\infty}}} \cong T_K(M^0)$ for $M^0$ as in Theorem~\ref{thm:gls-irregular}  with $\underline{r} = \BT$, although $\rhobar$ has no Barsotti--Tate lift.)
    
    For regular $\underline{r}$, Theorem~\ref{prop:modified-basis-of-invariants} and Theorem~\ref{thm:gls-irregular} therefore prove directly that the $\F'$-points of $\cZ^{\tau}(J)$ are precisely those admitting a crystalline lift of Hodge type $r(\tau,J)$, furnishing a new, purely local proof of a result from \cite{cegsA}.
\end{remark}

 In the irregular case, we can instead \emph{use} Theorem~\ref{thm:irregular-locus} to deduce the converse to Theorem~\ref{thm:gls-irregular}.
 
\begin{cor}\label{cor:gls-converse} Suppose $\rhobar$ is not a twist of a tr\'es ramifi\'ee representation. 
The ``if-then'' of Theorem~\ref{thm:gls-irregular} is an ``if-and-only-if'' when the Hodge type $\underline{r}$ is $p$-bounded. %\mar{HL: don't we need to say ``if-then'' of Thm 4.22? (as in Rem 4.24).}
\end{cor}
\begin{proof} 
As explained in Remark~\ref{rem:converse-remark-regular} it remains to prove the Corollary when $\underline{r}$ is irregular. Suppose more generally that $\underline{r}$ is non-Steinberg.  By Proposition \ref{prop:existence-tame-type}, 
 can find non-scalar $\tau$ and a profile $J$ such that $\cN^{\tau}(J)$ is the locus of mod $p$ representations with crystalline lifts of Hodge type $\underline{r}$. The converse statement follows from the equality $\cN^{\tau}(J) = \cZ^{\tau}(J)$ together with Theorem~\ref{prop:modified-basis-of-invariants}.
\end{proof}

\section{Inclusions between \texorpdfstring{$p$}{p}-bounded crystalline loci}\label{sec:geography}

%\mar{KK: Should we say something about the $f=1$ case? For instance, in the proof of Theorem 5.2, we could take the same C and D as for $\nu_j$ to get inclusions when $f=1$.} 

Suppose that the $p$-bounded Hodge type $\underline{r}$ is irregular. Since the locus $\cX^{\underline{r}}_{\red} \subset \cX_{2,\red}$ has positive codimension, it is reasonable to imagine that there exist inclusions \[ \cX^{\underline{r}}_{\red} \subset \cX^{\underline{r}'}_{\red}\] for certain other $p$-bounded Hodge types $\underline{r}'$. For example when $f=1$, so that $K=\Qp$, the locus $\cX^{\underline{0}}_{\red}$ of unramified representations is contained in the irreducible component of $\cX_{2,\red}$ associated to the Serre weight $\Sym^{p-2} \Fpbar^2$. Up to twist, this is the only proper inclusion of $p$-bounded crystalline loci when $K = \Qp$.

For the remainder of this section, we assume that $f \ge 2$, and we will prove in some additional situations  that representations having a crystalline lift with Hodge type $\underline{r}$ necessarily also have a crystalline lift with Hodge type $\underline{r}'$, and deducing as a corollary that $\cX^{\underline{r}}_{\red} \subset \cX^{\underline{r}'}_{\red}$.

In fact we  give two arguments. The first argument is short and direct, using Corollary~\ref{cor:gls-converse} and an explicit change of basis for \'etale $\varphi$-modules. The second argument, which we call ``shape-shifting'', is more geometric and (in our opinion) carries some explanatory power, but is also more complicated. The shape-shifting argument  relies on the observation that for the Hodge types $\underline{r}$ and $\underline{r}'$  under consideration, there exists a tame type $\tau$ and profiles $J,J'$ with $\underline{r} = r(\tau,J)$ and $\underline{r}' = r(\tau,J')$.

\subsection{The direct argument}

We begin by defining several operators on Hodge types. The first two can be viewed as analogues of partial theta operators and partial Hasse invariants in the work of Diamond and Sasaki on geometric Serre weight conjectures (\cite{DiamondSasaki}). %\mar{HW: not too sure about the wording here}

\begin{defn}\label{def:operators}
    Let $f\geq 2$. For each $j \in \Z/f\Z$, we define operators $\theta_j$, $\mu_j$, and $\nu_j$ on Hodge types $\underline{r}$
   by setting:
\begin{align*}
         \theta_{j}(\underline{r})_{{i}} &:= \left\{ \begin{array}{ll}
        (r_{1, {i}},r_{2, {i}}-1) & \text{if } {i}= {j}-1, \\
         (r_{1, {i}}+p,r_{2, {i}}) & \text{if } {i}= {j}, \\
         (r_{1, {i}},r_{2, {i}}) & \text{otherwise. }  
         \end{array} \right.
         \end{align*}
         \begin{align*}
    \mu_{j}(\underline{r})_{{i}} &:= \left\{ \begin{array}{ll}
        (r_{1, {i}}-1,r_{2, {i}}) & \text{if } {i}= {j}-1, \\
         (r_{1, {i}}+p,r_{2, {i}}) & \text{if } {i}= {j}, \\
         (r_{1, {i}},r_{2, {i}}) & \text{otherwise. }  
         \end{array} \right.  \end{align*}
         \begin{align*}
        \nu_{j}(\underline{r})_{{i}} &:= \left\{ \begin{array}{ll}
        (r_{1, {i}},r_{2, {i}}-1) & \text{if } {i}= {j}, \\
         (r_{2, {i}} + p,r_{1, {i}}) & \text{if } {i}= {j}+1, \\
         (r_{1, {i}},r_{2, {i}}) & \text{otherwise. }  
         \end{array} \right.
\end{align*}
\end{defn}

We now establish the following.

\begin{thm}\label{thm:operator-lifts-exist}
    Suppose that $\underline{r}$ is $p$-bounded and irregular, with $r_{1,j} = r_{2,j}$. If $\rhobar$ has a crystalline lift of Hodge type $\underline{r}$, then $\rhobar$ also has crystalline lifts of Hodge type $\mu_j(\underline{r})$, and $\nu_j(\underline{r})$, as well as of Hodge type $\theta_j(\underline{r})$ provided that $r_{1,j-1} - r_{2,j-1} \neq p$. 
\end{thm}

%\mar{HW: should we be clear that we fix $j$? And perhaps mention Lemma 11.2.6 of Diamond-Sasaki?}
Note that $\mu_{j}(\underline{r})$ and $\nu_j(\underline{r})$ in the statement of the theorem are still $p$-bounded, and the hypothesis $r_{1,j-1} - r_{2,j-1} \neq p$ guarantees the same for $\theta_j(\underline{r})$.

\begin{proof}
    The proof of Lemma~\ref{lem:p-bounded-ns} shows that $\rhobar$ is not a twist of a tr\'es ramifi\'ee representation. By Corollary~\ref{cor:gls-converse}, it suffices to prove that if $M^0 \in \cR_2(\F')$ has a basis $\beta = (\beta_i)_{i \in \Z/f\Z} = (x_i,y_i)_{i \in \Z/f\Z}$ such that the partial Frobenius maps $\Phi_{M^0,i}$, written with respect to $\beta$,
have matrices
\[
    B_{i} \begin{pmatrix}
    v^{r_{i,1}}  & 0 \\
    0 & v^{r_{i,2}}
    \end{pmatrix}
\]
for some $B_i \in \GL_2(\F'[\![v]\!])$, then the same holds with $r_{i,1},r_{i,2}$ replaced by $r'_{i,1},r'_{i,2}$ for all $i$, for each $\underline{r}' \in \{ \mu_j(\underline{r})$,  $\nu_j(\underline{r})$,  $\theta_j(\underline{r}) \}$. (The hypothesis $r_{1,j-1} - r_{2,j-1} \neq p$ is necessary to apply Corollary~\ref{cor:gls-converse} when $\underline{r}' = \theta_j(\underline{r})$, but not for any other part of the argument.) The key is simply that, since $r_{1,j} = r_{2,j}$, the matrix $\left(\begin{smallmatrix}
    v^{r_{j,1}}  & 0 \\
    0 & v^{r_{j,2}}
    \end{smallmatrix}\right)$ is scalar and therefore lies in the center of the matrix ring $M_2(\F'(\!(v)\!))$. 
Consider the basis $\beta'$ in which 
\[ \beta'_{j-1} = \beta_{j-1} B C, \quad \beta'_j = \beta_{j} D \]
and $\beta'_i = \beta_i$ if $i \neq j-1,j$. Here  $B,C,D \in \GL_2(\F'(\!(v)\!))$ are matrices to be chosen momentarily such that
\begin{itemize}
    \item $B \in \GL_2(\F'[\![v]\!])$, 
    \item $C,D$ are diagonal, and
    \item $C$ commutes with $B^{-1} B_{j-1}$.
\end{itemize}
Then the matrices of $\Phi_{M^0,i}$ with respect to $\beta'$ for $i = j-1, j, j+1$ are  checked to be %\mar{RB: typesetting?}
\begin{multline*}
    B^{-1} B_{j-1}
    \begin{pmatrix}
    v^{r_{j-1,1}}  & 0 \\
    0 & v^{r_{j-1,2}}
    \end{pmatrix} C^{-1} , \\
     D^{-1} B_j \varphi(B) D \begin{pmatrix}
    v^{r_{j,1}}  & 0 \\
    0 & v^{r_{j,2}}
    \end{pmatrix} D^{-1} \varphi(C), \quad \textrm{and} \\        B_{j+1} \begin{pmatrix}
    v^{r_{j+1,1}}  & 0 \\
    0 & v^{r_{j+1,2}} 
    \end{pmatrix}   \varphi(D)
\end{multline*}  
respectively. The theorem follows by Remark~\ref{rem:again-order-irrelevant}, choosing $B,C,D$ as follows:
\begin{itemize}
    \item For $\theta_j$, take $B = B_{j-1}$, $C = \left(\begin{matrix}
    1  & 0 \\
    0 & v
    \end{matrix}\right)$ and $D = I$, 
    \item For $\mu_j$, take $B = B_{j-1}$,  $C = \left(\begin{matrix}
    v  & 0 \\
    0 & 1
    \end{matrix}\right)$ and $D = I$, 
    \item  For $\nu_j$, choose $B = B_j^{-1}$, so that $B_j \varphi(B) \equiv I \pmod{v}$, and take $C=I$ and $D = \left(\begin{matrix}
    1  & 0 \\
    0 & v
    \end{matrix}\right)$.
\end{itemize}
\end{proof}

\begin{cor}
        Suppose that $\underline{r}$ is $p$-bounded and irregular, with $r_{1,j} = r_{2,j}$. Then \[ \cX^{\underline{r}}_{\red} \subset \cX^{\underline{r}'}_{\red}\] for each $\underline{r}' \in \{\mu_j(\underline{r}), \nu_j(\underline{r}), \theta_j(\underline{r}) \}$,  provided if $\underline r'=\theta_j(\underline r)$ that we additionally assume  $r_{1,j-1} - r_{2,j-1} \neq p$, so that $\underline{r}'$ is $p$-bounded.
\end{cor}
%\mar{RB: I would reword the final clause as something like "; if $\underline r'=\theta_j(\underline r)$, we additionally assume that $r_{1,j-1} - r_{2,j-1} \neq p$." DS: done!}

\begin{remark} It is natural to ask whether $\cX^{\underline{r}}_{\red}$ is equal to the reduced intersection of the irreducible components of $\cX_{2,\red}$ that contain it. For example, suppose that $f=2$, that $\underline{r}$ is irregular at $i=1$, and that $0 < r_{0,1} - r_{0,2} < p$.  Then \cite[Lem.~11.2.6]{DiamondSasaki} has the following geometric reinterpretation:\ if  $r_{0,1} - r_{0,2} \neq 1$, then $\cX^{\underline{r}}_{\red} = \cX^{\mu_1(\underline{r})}_{\red} \cap \cX^{\theta_1(\underline{r})}_{\red}$;\ if instead $r_{0,1} - r_{0,2} = 1$, then the  $\cX^{\underline{r}}_{\red} = \cX^{\mu_0(\mu_1(\underline{r}))}_{\red} \cap \cX^{\theta_1(\underline{r})}_{\red}$. Note that in the latter case $\mu_1(\underline{r})$ is irregular at $i=0$, and moreover $\nu_0(\mu_1(\underline{r})) \sim \underline{r}$, so that $\cX^{\underline{r}}_{\red} = \cX^{\mu_1(\underline{r})}_{\red}$. (The intersections here are all reduced intersections, i.e., we make no claims about intersection multiplicities.)

For general $f$, an explicit conjecture in the same spirit can be found in \cite[Conj.~4.4]{HWThesis}, and will be addressed in forthcoming work of Wiersema.

\end{remark}
         
\subsection{Shape-shifting}

We now give another proof of Theorem~\ref{thm:operator-lifts-exist}. The strategy is as follows.  By Proposition~\ref{prop:existence-tame-type} we can find a tame type $\tau$ and a profile $J$ such that $\underline{r} = r(\tau,J)$. Recall that the set $S^{\tau}(J)$ is precisely the set of embeddings at which $\underline{r}$ is irregular.

The substack $\cZ^{\tau}(J)$ is the scheme-theoretic image in $\cZ^{\dd,1}$ of $\cC^{\tau}(J)$, which by the results of Section~\ref{sec:shape} is the stack of Breuil--Kisin modules of type $\tau$ that have shape $\txI_\eta$ or $\txII$ when $i \in J$, and shape $\txI_{\eta'}$ or $\txII$ when $i \not\in J$. 

The key observation will be that if $U \subset \cC^{\tau}(J)$ is the closed substack of Breuil--Kisin modules with shape $\txII$ for each $i \in S^\tau(J)$, then the image of $U$ in $\cZ^{\tau}(J)$ is still dense (\emph{cf.}\ the proof of Theorem~\ref{thm:shape-shifting}). Since Breuil--Kisin modules of shape~$\txII$ may be regarded as having either ``shape $\txI_{\eta}$ or $\txII$'' or ``shape $\txI_{\eta'}$ or $\txII$'', we see that  if $J'$ is any other profile such that the symmetric difference $J \triangle J'$ is a subset of $S^{\tau}(J)$, then 
 $U \subset \cC^{\tau}(J')$ as well. (This observation is the source of the name \emph{shape-shifting}.) It follows that $\cZ^{\tau}(J)$ is contained in $\cZ^{\tau}(J')$. Finally, setting $\underline{r'} = r(\tau,J')$, Theorem~\ref{thm:irregular-locus} and the results of Section~\ref{sec:comparisons} imply that $\cX^{\underline{r}}_{\red} \subset \cX^{\underline{r}'}_{\red}$. Note that in this argument the inclusion of stacks comes first, and the statement about crystalline lifts is the corollary.

 In fact the collection of profiles~$J'$ to which the shape-shifting argument can be made to apply is somewhat larger than we have described above.

\begin{remark}%\mar{RB: The second sentence here is worded a bit oddly}
    In the cuspidal case,  each profile $J \subset \Z/f'\Z$ has the property that $i \in J$ if and only if $i+f \not \in J$.  It follows that the symmetric difference $J \triangle J' \subset \Z/f'\Z$ of two profiles has the property that $i \in J \triangle J'$ if and only if $i+f \in J \triangle J'$, and may thus be identified with a well-defined subset of $\Z/f\Z$. We will freely make this identification in what follows. This allows us sensibly to write $J \triangle J' \subset S^{\tau}(J)$  in the cuspidal case and not only in the principal series case, even though in the cuspidal case $J \triangle J'$ is literally a subset of $\Z/f'\Z$ while $S^\tau(J)$ is a subset of $\Z/f\Z$. 
\end{remark}

To implement the above strategy, we begin with a brief review of some results from \cite[\S\S3--5]{cegsC}. As we have already alluded to in the proof of Proposition~\ref{prop:dim-Z-tau-J},  the constructions of \cite{cegsC} furnish us with a morphism
\[ \xi : \Spec B^{\dist} \to \cC^{\tau}(J) \to \cZ^{\tau}(J) \]
such that the maps from $\Spec B^{\dist}$ to both $\cC^{\tau}(J)$ and $\cZ^{\tau}(J)$ are scheme-theoretically dominant. The source $\Spec B^{\dist}$ has the following description:\ there are rank one Breuil--Kisin modules $\gM(J)$ and $\gN(J)$ such that $\Spec B^{\dist}$ is a universal family of extensions of unramified twists of $\gM(J)$ by unramified twists of $\gN(J)$;\ the superscript `$\dist$' indicates that for certain $(\tau,J)$ --- namely if $\gM(J)[1/u] \cong \gN(J)[1/u]$ --- then we restrict from the whole universal family to the (dense, open) subfamily whose $\F'$-points are extensions of $\gM(J)_{\F',a}$ by $\gN(J)_{\F',b}$ with $a \neq b$. 

The rank one Breuil--Kisin modules $\gM(J)_{\F',a}$ and $\gN(J)_{\F',b}$ admit the following descriptions. Set $(c_i,d_i) = (k_i,k'_i)$ if $i \in J$, and $(c_i,d_i) = (k'_i,k_i)$ if $i \not \in J$. Define 
\begin{align*}
r_i = \begin{cases}
[d_i - c_i] &\text{ when } (i-1, i) \text{ is a transition},\\
p^{f'} - 1 &\text{ when } (i-1, i) \text{ is not a transition}.
\end{cases}\\
s_i = \begin{cases}
[c_i -d_i] &\text{ when } (i-1, i) \text{ is a transition},\\
0 &\text{ when } (i-1, i) \text{ is not a transition}.
\end{cases}
\end{align*}
Finally set $a_0 = a$, $b_0 = b$, and $a_i = b_i = 1$ if $i \neq 0$. Then $\gM(J)_{\F',a}$ is the Breuil--Kisin module $\gM(r,a,c)$ of \cite[Lem~4.1.1]{cegsC} (with $\F'$-coefficients), and $\gN(J)_{\F',b}$ is the Breuil--Kisin module $\gM(s,b,c)$.  In particular, the $i$-th component $(\gM(J)_{\F',a})_i$ has a basis element $m_i$ on which $I(K'/K)$ acts via $\eta$ if $i \in J$ and $\eta'$ if $i \not\in J$, while the reverse holds for basis elements $n_i$ of  $(\gN(J)_{\F',b})_i$.

As explained in \cite[Rem.~4.1.9]{cegsC}, an extension $\gP$ of $\gM(J)_{\F',a}$ by $\gN(J)_{\F',b}$ has partial Frobenius given by 
\begin{align*}
\Phi_{\gP,i}(1 \otimes n_{i-1}) &= b_i u^{s_i} n_i ,\\
\Phi_{\gP,i}(1 \otimes m_{i-1}) &= a_i u^{r_i} m_i + h_i u^{\delta_i} n_i,
\end{align*}
where $h_i \in \F'$ and $\delta_i = 0$ if $(i-1,i)$ is a transition while $\delta_i = [c_i-d_i]$ otherwise. The descent data on $\gP$ is given by specifying that if we define
$\beta_i = (m_i,n_i)$ for $i \in J$ and $\beta_i = (n_i,m_i)$ for $i \not\in J$, then $\beta = (\beta_i)$ is an eigenbasis. Observe for later reference that $\gP$ has shape $\txII$ at $i$ if and only $(i-1,i)$ is a transition and $h_i = 0$ (in which case the matrix of $\Phi_{\gP,i}$ with respect to $\beta$ is anti-diagonal). In this manner we identify $\Ext^1(\gM(J)_{\F',a}, \gN(J)_{\F',b})$ with the $f$-dimensional vector space spanned by the elements $h_i \in \F'$.

To describe the subspace \[ \kExt^1(\gM(J)_{\F',a}, \gN(J)_{\F',b}) \subset \Ext^1(\gM(J)_{\F',a}, \gN(J)_{\F',b})\] of extensions that split after inverting $u$, we need to introduce some notation.

\begin{defn}
An \emph{interval} in $\Z/f\Z$ is the image in $\Z/f\Z$ of any interval in $\Z$.    If $S$ is any subset of $\Z/f\Z$ write $S(n)$ for the shift of $S$ by $n$, and $S^e = S \cup S(-1)$. Any subset $S \subset \Z/f\Z$ then has a unique decomposition $S = I_1 \coprod \cdots \coprod I_\ell$ as a disjoint union of maximal intervals. The maximality condition is equivalent to the condition that $I_i^e \cap I_j^e = \varnothing$ for all $i \neq j$.
\end{defn}

The discussion in \cite[\S5.1]{cegsC} establishes that $\kExt^1(\gM(J)_{\F',a}, \gN(J)_{\F',b})$ has dimension $|S^\tau(J)|$, and in fact that it has the following more precise description. Note that the assumption that $a\neq b$ if 
if $\gM(J)[1/u] \cong \gN(J)[1/u]$ implies that we are not in the ``exceptional case'' of \cite[Prop.~5.1.8]{cegsC}.
In what follows, we let $I_1 \coprod \cdots \coprod I_\ell$ be the decomposition of $S^{\tau}(J)$ as a disjoint union of maximal intervals. (For simplicity of notation, we suppress $\tau, J$ from the notation for $I_1,\ldots,I_\ell$.)

\begin{prop}\label{prop:ker-ext}
Suppose that $S^\tau(J) \neq \Z/f\Z$. For each $k = 1,\ldots,\ell$ there is a hyperplane $V_k = \{ \sum_{i \in I_k^e} \alpha_i h_i = 0 \}$  with each $\alpha_i \neq 0$ such that 
\[ \kExt^1(\gM(J)_{\F',a} \gN(J)_{\F',b}) \cong \bigoplus_{k=1}^\ell V_k \]
under the identification discussed above.

If instead $S^\tau(J) = \Z/f\Z$ then $\kExt^1(\gM(J)_{\F',a} \gN(J)_{\F',b})$ is equal to  all of $\Ext^1(\gM(J)_{\F',a}, \gN(J)_{\F',b})$.
\end{prop}

\begin{defn}
Continue to let $I_1 \coprod \cdots \coprod I_\ell$ be the decomposition of $S^{\tau}(J)$ as a disjoint union of maximal intervals as above. If $I \neq \Z/f\Z$ is an interval let $m(I)$ be the unique element of $I^e \setminus I$, and set 
\[ I' = \begin{cases} I^e & \text{if } (m(I)-1,m(I)) \text{ is a transition for $J$} \\
I & \text{if } (m(I)-1,m(I)) \text{ is not a transition for $J$}.
\end{cases}
\]
If $I = \Z/f\Z$ set $I' = I$. Define $S^{\tau}(J)' = I'_1 \coprod \cdots \coprod I'_\ell$.
\end{defn}

We can now prove the following.

\begin{thm}\label{thm:shape-shifting}
Let $J'$ be any profile with $J \triangle J' \subset S^\tau(J)'$. If $S^{\tau}(J) \neq \Z/f\Z$, assume further that $J \triangle J'$ does not contain $I^e_k$ for any $k$.

Write $\underline{r} = r(\tau,J)$ and $\underline{r}' = r(\tau,J')$. Then $\cX^{\underline{r}}_{\red} \subset \cX^{\underline{r}'}_{\red}$.
\end{thm}

\begin{proof}
Let $V \subset \cZ^{\tau}(J)$ be the collection of finite type points lying in the image of $\Ext^1(\gM(J)_{\F',a}, \gN(J)_{\F',b})$ for some $a,b$ with $a\neq b$, so that $V$ is dense in $\cZ^\tau(J)$. Proposition~\ref{prop:ker-ext} implies that each point in $V$ has a preimage $\gP \in \Ext^1(\gM(J)_{\F',a}, \gN(J)_{\F',b})$ with $h_i = 0$ for all $i \in J \triangle J'$ (here using the fact that $\alpha_i \neq 0$ for all $i$, and the hypothesis that $I^e_k \not\subset J\triangle J'$ for any $k$). By the discussion at the beginning of the section, this preimage has shape $\txII$ at all $i \in J\triangle J'$, here using in an essential way that $i = m(I_k) \in I_k'$ is allowed only when $(m(I)-1,m(I))$ is a transition. Therefore $\gP \in \cC^{\tau}(J')(\F')$. It follows that $V \subset \cZ^{\tau}(J')$, and so also $\cZ^{\tau}(J) \subset \cZ^{\tau}(J')$.
\end{proof}

\begin{cor}
    Theorem~\ref{thm:operator-lifts-exist} holds  for  $\nu_j(\underline{r})$ and $\theta_j(\underline{r})$.
\end{cor}

\begin{proof}
   Using Proposition~\ref{prop:existence-tame-type} choose $\tau$ and $J$ so that $\underline{r} = r(\tau,J)$. By hypothesis $r_{j,1} = r_{j,2}$ and so we have $j \in S^\tau(J)$. Taking the unique profile $J'$ with $J \triangle J' = \{j\}$, one computes that $\nu_j(\underline{r}) \sim r(\tau,J')$. The result for $\nu_j$ now follows from Theorem~\ref{thm:shape-shifting}.
   %\mar{DS: from the point of view of this calculation it is more natural to define $\nu_j$ with the pairs $(r_{1,j}+1,r_{2,j})$ and $(r_{2,j+1},r_{1,j}-p)$. HW: this would be fine by me! DS: what Diamond-Sasaki call $\theta$ seems to be what we currently have, so I suggest leaving as is} 
   
   The argument for $\theta_j(\underline{r})$ is similar but slightly more involved. If $j-1 \in S^\tau(J)$ then $\theta_j(\underline{r}) = \nu_{j-1}(\underline{r})$ and we are done by the previous paragraph. Otherwise, since we have assumed that $r_{1,j-1} - r_{2,j-1} \neq p$, we have $r_{1,j-1}-r_{2,j-1} \in [1,p-1]$. By Remark~\ref{rem:cant-choose-j-when} we may choose $(\tau,J)$ with $\underline{r} = r(\tau,J)$  such that there is a transition at $j-1$, unless we are in the exceptional case described in Remark~\ref{rem:cant-choose-j-when}(2). Observe that this exceptional case occurs precisely when $\theta_j(\underline{r})$ is Steinberg.

 Assume first that $\theta_j(\underline{r})$ is non-Steinberg. Then as explained above we may arrange that there is a transition at $j-1$, so that $j-1 \in S^\tau(J)'$. Taking the unique profile~$J'$ with $J \triangle J' = \{j-1\}$, one computes that $\theta_j(\underline{r}) \sim r(\tau,J')$. The result in this case now follows from Theorem~\ref{thm:shape-shifting}.
 
% \mar{KK: I think the verification that $\nu_j(r) \sim r(\tau, J')$ (resp. $\theta_j(r) \sim r(\tau, J')$) is not immediate when $\tau$ is cuspidal, because, a priori, one has to compute the changes in $\theta_{J,i}$'s and not just $s_{J,i}$'s and $t_{J,i}$'s. I am completely okay with the way things are right now, but I wonder if it's worthwhile to share some of those computations. DS: KK and I discussed and are omitting this}

 Finally suppose that $\theta_j(\underline{r})$ is Steinberg.  
 Applying $\nu_{j+1},\nu_{j+2},\ldots,\nu_{j-1}$ successively to $\underline{r}$ one obtains a Hodge type of the form $\BT + \underline{\lambda}$ with $\underline{\lambda} \in \Z^f$, and such that $\theta_j(\underline{r}) \sim \St + \underline{\lambda}$. Here $\St$ is the Steinberg Hodge type $\{p,0\}_{i\in \Z/f\Z}$. The result for $\nu$ shows that $\rhobar$ has a crystalline lift of Hodge type $\BT + \underline{\lambda}$. Since it is standard that a representation with a crystalline lift of Hodge type $\BT$ also has a crystalline lift of Hodge type $\St$, the result follows.
\end{proof}

The argument for $\mu_j$ proceeds somewhat differently. For brevity, since the argument for Theorem~\ref{thm:operator-lifts-exist} in the previous subsection was complete, we content ourselves with giving a sketch.

\begin{prop}
    Theorem~\ref{thm:operator-lifts-exist} holds  for  $\mu_j(\underline{r})$.
\end{prop}

\begin{proof}[Sketch of proof.]
If $\underline{r}$ is irregular at $j-1$ then $\mu_j = \nu_{j-1}$, so we may assume that $\underline{r}$ is regular at $j-1$. Using Proposition~\ref{prop:existence-tame-type} choose $\tau$ and $J$ so that $\underline{r} = r(\tau,J)$. By Remark~\ref{rem:cant-choose-j-when} we can always arrange that $j-1$ is not a transition, and we do so. Let $J'$ be the unique profile such that $J \triangle J' = \{j-1\}$. Then $r(\tau,J') \sim \mu(\underline{r})$. Note that $|S^{\tau}(J')| = |S^{\tau}(J)|-\delta$ where $\delta = 1$ if $r_{1,j-1}-r_{2,j-1} > 1$ and $\delta = 0$ otherwise. In the latter case $j-1 \in S^\tau(J').$

Unfortunately we cannot apply Theorem~\ref{thm:shape-shifting} to the pair $J,J'$, because $j-1 \not\in S^{\tau}(J)'$. Instead we argue as follows. The extensions of $\gM(J')_{\F',a}$ by $\gN(J')_{\F',b}$ have a description that is parallel to the one for $\gM(J)_{\F',a}$ by $\gN(J)_{\F',b}$. We let $h'_i$ denote the extension parameters for $J'$ that were denoted $h_i$ for $J$.  Since $j-1$ \emph{is} a transition in $J'$, the locus $\{ h'_{j-1} = 0 \} \subset \cC^{\tau}(J')(\F')$ consists of Breuil--Kisin modules having shape $\txII$ at $j-1$,  which therefore also lie in $\cC^{\tau}(J)(\F')$. 
%\mar{KK: should we say shape $\txII$ at $j-1$ instead of just shape $\txII$?}

If $\delta = 1$ then $j-1,j \not\in S^\tau(J')$ and according to Proposition~\ref{prop:ker-ext} we have 
$\kExt(\gM(J')_a,\gN(J')_b) \subset \{h'_{j-1} = 0\}$.
If  $\delta = 0$ then $\kExt(\gM(J')_a,\gN(J')_b)$ meets $\{h'_{j-1} = 0\}$
in codimension $1$ instead. In either case $\kExt(\gM(J')_a,\gN(J')_b) \cap \{h'_{j-1} = 0\}$ has dimension $|S^\tau(J')| + \delta - 1 = |S^\tau(J)| - 1$.

%\mar{KK: While reading this, I justified ``dimension at least $[K:\Q] - |S^\tau(J)|$" by first justifying that the dimension was exactly $[K:\Q] - |S^\tau(J)|$. Is it more obvious to see ``at least"?} 
The image of the locus $\{h'_{j-1}=0\}$ in $\cZ^{\tau,1}$ will therefore have dimension  $([K:\Q] - 1) - (|S^\tau(J)| - 1)  = [K:\Q] - |S^\tau(J)|$. It follows that $\cZ^{\tau}(J')$ contains a ($[K:\Q] - |S^\tau(J)|$)-dimensional subset of $\cZ^\tau(J)$. But $\cZ^\tau(J)$ is irreducible of dimension $[K:\Q] - |S^\tau(J)|$;\ after checking that this implies $\cZ^\tau(J') \cap \cZ^\tau(J)$ is dense in $\cZ^\tau(J)$,  we conclude that $\cZ^\tau(J) \subset \cZ^{\tau}(J')$.
\end{proof}

\bibliographystyle{math} %Bibliography style file X.bst
\bibliography{apaw} % Bibliography database file Y.bib
\end{document}